\documentclass[final,onefignum,onetabnum]{siamltex1213}
\usepackage{amssymb,amsmath}
\newtheorem{remark}[theorem]{Remark}
\newtheorem{example}[theorem]{Example}

\title{Perturbations of time optimal control problems\\ for a class of abstract parabolic systems\thanks{We acknowledge the
support of the Sino-French International Associated Laboratory for Applied Mathematics (LIASFMA) which allowed the three months visit of the second author in Institut \'Elie Cartan of Lorraine. The second author was  partially supported by the National Basis Research Program of China under grant 2011CB808002 and the NNSF of China under grant 11171264.
}}

\author{Marius Tucsnak\thanks{Universit\'e de Lorraine, Institut \'Elie Cartan de Lorraine, UMR 7502, 54506 Vandoeuvre-l\`es-Nancy Cedex, France
(\email{marius.tucsnak@univ-lorraine.fr}, \email{chi-ting.wu@univ-lorraine.fr})} \and Gensgsheng Wang\thanks{Department of Mathematics and Statistics, Wuhan University, Wuhan,
China (\email{wanggs62@yeah.net})} \and Chi-Ting Wu\footnotemark[2] }

\begin{document}
\maketitle
\slugger{mms}{sicon}{xx}{x}{x--x}

\begin{abstract}
In this work we study the asymptotic behavior of the solutions of a class of abstract parabolic time optimal control problems
when the generators converge, in an appropriate sense, to a given strictly negative operator.
Our main application to PDEs systems concerns
the behavior of optimal time and of the associated optimal controls for parabolic equations with highly oscillating coefficients,
as we encounter in homogenization theory. Our main results assert that, provided that the target
is a closed ball centered at the origin and of positive radius, the solutions of the time optimal
control problems for the systems with oscillating coefficients converge, in the usual norms,
to the solution of the corresponding problem for the homogenized system. In order to prove our main theorem,
we provide several new results, which could be of a broader interest, on time and norm optimal control problems.
\end{abstract}

\begin{keywords}
time optimal control, norm optimal control, abstract parabolic equation with perturbed parameter,
heat equation with oscillating coefficient, homogenization
\end{keywords}

\begin{AMS}93C20; 93C25;  93C73; 35K90 ;  35K10. \end{AMS}

\pagestyle{myheadings}
\thispagestyle{plain}
\markboth{Tucsnak, Wang and Wu}{Perturbations of time optimal control problems}

\section{Introduction}\label{sec_intro}
\setcounter{equation}{0}

\

The aim of this paper is to study the behavior of time optimal controls for a family of systems governed
by abstract parabolic equations in a Hilbert space. The systems depend on a small parameter
$\varepsilon$ and they converge, in an appropriate sense, to a given abstract parabolic linear system when $\varepsilon \to 0$.
Our motivation comes from the need of understanding the behavior of time optimal controls for
 parabolic equations with rapidly oscillating coefficients, as studied in homogenization theory.
 To state  this motivating problem, we consider an open bounded set $\Omega\subset\mathbb{R}^n$
 with a $C^2$ boundary $\partial\Omega$ and  an open nonempty subset $\omega\subset \Omega$,
 with its characteristic function
 $\chi_\omega$.
Let $a=(a_{ij})_{1\leqslant i, j\leqslant n}\in W^{2,\infty}(\mathbb{R}^n; M_n(\mathbb{R}))$
 be a symmetric  matrix valued function of period 1 in each variable $x_j$, with $j\in \{1,\dots,n\}$
 such that, for some constants $m_1, m_2>0$, we have
$$
m_1 \|\xi\|^2 \leqslant a(x)\xi\cdot \xi\leqslant m_2 \|\xi\|^2 \qquad(\xi\in \mathbb{R}^n,\ x\in\mathbb{R}^n\ {\rm a.e.}).
$$
Given $u\in L^\infty((0,\infty),L^2(\Omega))$, $\psi\in L^2(\Omega)$ and $\varepsilon>0$, we
consider the following initial and boundary value problem:
\begin{equation}\label{heat_osc}
\dot z_\varepsilon(x,t)-{\rm div}\, \left(a\left(\frac{x}{\varepsilon}\right)\nabla z_\varepsilon(x,t)\right)
= \chi_\omega(x) u(x,t) \qquad(x\in \Omega,\ \ t\geqslant 0),
\end{equation}
\vspace{-0.75cm}
\begin{equation}\label{bound_osc}
z_\varepsilon(x,t)=0 \qquad(x\in\partial\Omega,\ \ t\geqslant 0),
\end{equation}
\vspace{-0.75cm}
\begin{equation}\label{init_osc}
z_\varepsilon(x,0)=\psi(x) \qquad(x\in \Omega).
\end{equation}
We treat the solution of the above system as a function from $[0,\infty)$ to $L^2(\Omega)$, and denote it by
$z_\varepsilon(\cdot; \psi, u)$.
It is well known (see, for instance, Bensoussan {\it et al} \cite{BLP}) that for given $u\in L^\infty((0,T);L^2(\Omega))$
and $\psi\in L^2(\Omega)$ the solution $z_\varepsilon(\cdot; \psi, u)$ of \eqref{heat_osc}-\eqref{init_osc}
converges, in a sense which we will make precise later, to the solution $z(\cdot; \psi, u)$
of:
\begin{equation}\label{heat_osc_non}
\dot z (x,t)-{\rm div}\, \left(a_0 \nabla z(x,t)\right)= \chi_\omega(x) u(x,t) \qquad(x\in \Omega,\ \ t\geqslant 0),
\end{equation}
\vspace{-0.75cm}
\begin{equation}\label{bound_osc_non}
z(x,t)=0 \qquad(x\in\partial\Omega,\ \ t\geqslant 0),
\end{equation}
\vspace{-0.75cm}
\begin{equation}\label{init_osc_non}
z(x,0)=\psi(x) \qquad(x\in \Omega),
\end{equation}
where $a_0\in M_n(\mathbb{R})$ is a positive definite matrix, called the {\em homogenization matrix}, which will be precisely defined in Section \ref{sec_final}.

To state our motivating problem more precisely, we fix $r>0$ and $\psi\in L^2(\Omega)$.  Consider, for each
 $\varepsilon>0$ and $ M>0$, the following time optimal control problems associated to
 systems \eqref{heat_osc}-\eqref{init_osc} and \eqref{heat_osc_non}-\eqref{init_osc_non}, respectively:

\

\noindent $(TPH)^M_\varepsilon\;\;\;\;
\tau^*_{\varepsilon}(M):= \displaystyle\min \left\lbrace t \ |\
\|z_\varepsilon(t; \psi, u)\|_{L^2(\Omega)}\leqslant r\right\rbrace $,
 where the minimum  is taken over all $u$ from the closed ball in $L^\infty((0,\infty); L^2(\Omega))$,
 centered at $0$ and of  radius $M$;

\

\noindent ${(TPH)_0^M}\;\;\;\;
\tau_0^*(M):= \displaystyle\min \left\lbrace t \ |\ \|z(t; \psi,u)\|_{L^2(\Omega)}\leqslant r\right\rbrace$,
  where the minimum is taken over all $u$ from the closed ball in $L^\infty((0,\infty); L^2(\Omega))$,
   centered at $0$ and of  radius $M$.

 \

\noindent For fixed $\varepsilon>0$, time optimal control problems similar to $(TPH)^M_\varepsilon$ have been extensively studied in the literature, see, for instance, Tr{\"o}ltzsch \cite{Fredi}, Arada and Raymond \cite{Ar_Ray} or  Kunisch and Wang \cite{Kun_Wang}. It has been proved (see, for instance, Phung, Wang and Zhang \cite{PhWZh})  that $(TPH)^M_\varepsilon$ admits, for each $\varepsilon\geqslant 0,\ M>0$, a unique solution $(\tau_\varepsilon^*(M),u_\varepsilon)$. We refer to Barbu \cite{b_book},
 Fattorini \cite{fattorini_old_book,Fatto_book} and Li and Yong \cite{liyong} for detailed presentations of time optimal control problems for infinite-dimensional systems.

On the other hand, in recent years several authors studied the behavior of null or approximate controls for
\eqref{heat_osc}-\eqref{init_osc} versus the corresponding controls for
\eqref{heat_osc_non}-\eqref{init_osc_non} when $\varepsilon \to 0$.
Some references tackling these subjects for parabolic equations are Castro and Zuazua \cite{Castro-Zuaz_Proc}, Lopez and Zuazua \cite{Lop_Zuaz}, Tebou \cite{Tebou} and Zuazua \cite{Zuazua}. In the case of time-reversible linear PDEs, similar problems have been tackled in Saint-Jean Paulin and Vanninathan \cite{Jeanine_Vanni_94,Jeanine_Vanni_97} and Cioranescu and Vanninathan \cite{Doina_Vanni_98}.

The difficulty of the problem we consider comes from the high frequency
oscillations in the coefficients which might give blow-up of the controls when $\varepsilon \to 0$.
We here are  interested in   the behavior of time optimal controls for \eqref{heat_osc_non}-\eqref{init_osc_non} when $\varepsilon \to 0$.
The new challenges brought in by this problem are that the controls have to be found in
a bounded set of the space  $L^\infty((0,\infty); L^2(\Omega))$ and that both the optimal time and the
 optimal control strongly depends on $\varepsilon$.

As a consequence of our main theorem, which will be precisely stated later, we
will prove in Section \ref{sec_final} the following result:

\

\begin{proposition}\label{corollary6.4} With the above notation, for each $\varepsilon\geqslant 0$, let
$(\tau_\varepsilon^*(M), u^*_{\varepsilon}(M))$ be the solution of the time optimal control problem ${(TPH)^M_\varepsilon}$.
Then we have that
$\displaystyle \tau^*_\varepsilon(M) \to \tau_0^*(M)$, $\displaystyle  u^*_{\varepsilon} \to u_0^*$ strongly in $ L^2((0,\tau_0^*(M));L^2(\Omega))$ and $u^*_{\varepsilon} \to u_0^*$ strongly in $L^\infty((0,\tau_0^*(M)-\delta);L^2(\Omega))$, for every $\delta\in (0,\tau^*_0(M))$.
\end{proposition}

\

To prove the above result we found that it is  more convenient to give our main results in an abstract
framework, containing homogenization of time optimal control problems of parabolic equations as
a particular case. Besides providing more generality, this approach seems more appropriate to point out
the main ideas used in the proof.
We describe below the abstract framework which will be used in most of the remaining part of this paper. Let $X$ be a Hilbert space, whose inner product and the corresponding norm are  denoted by $\langle\cdot,\cdot\rangle$ and  $\|\cdot\|$ respectively. We introduce a family $(A_\varepsilon)_{\varepsilon\geqslant 0}$
of linear operators on $X$, with a common domain $X_1\subset X$, satisfying the following assumptions: \\ \ \\
(C1) The embedding $X_1\subset X$ is compact and the operator  $A_\varepsilon$ is,
for every $\varepsilon \geqslant 0$, self-adjoint in $X$.\\
(C2) $A_0$ is strictly positive and there exist positive constants $m_0$ and $m_1$ such that
$$
m_0 \langle A_0 \psi,\psi\rangle \leqslant \langle A_\varepsilon \psi,\psi\rangle \leqslant m_1
 \langle A_0 \psi,\psi\rangle \qquad(\psi\in X_1).
$$
(C3) We have $\displaystyle \lim_{\varepsilon \to 0+} \|A_\varepsilon^{-1} \psi- A_0^{-1}\psi\|=0$ for every $\psi\in X$.
\\ \ \\
It is well known that $-A_\varepsilon$ generates (for every $\varepsilon \geqslant 0$) a contraction semigroup $\mathbb{T}^\varepsilon$ on $X$, which is analytic (see, for instance, Tucsnak and Weiss \cite[Proposition 3.8.5]{obsbook}
and Arendt \cite[Example 3.7.5]{Cauchy}).
Let $U$ be another Hilbert space and let $B\in {\cal L}(U,X)$, called the control operator. Consider the family of systems $(\Sigma_\varepsilon)_{\varepsilon\geqslant 0}$, with the state space $X$ and the input space $U$, described by
$$
\dot z_\varepsilon(t)+A_\varepsilon z_\varepsilon(t)=Bu(t) \qquad (t\geqslant 0),\qquad
\qquad z_\varepsilon(0)=\psi\in X.  \leqno{(\Sigma_\varepsilon)}
$$
For $\varepsilon,\ \tau\geqslant 0$, we introduce the input to state operator
$\Phi_\tau^\varepsilon\in {\cal L}(L^\infty((0,\infty);U),X)$, defined by
\begin{equation}\label{fieps}
\Phi_\tau^\varepsilon u = \int_0^\tau \mathbb{T}_{\tau-\sigma}^\varepsilon Bu(\sigma)\, {\rm d}\sigma
\qquad (u\in L^\infty((0,\infty);U)),
\end{equation}
so that the state trajectory of $(\Sigma_\varepsilon)$ is given by
$z_\varepsilon(t)=\mathbb{T}_t^\varepsilon \psi+ \Phi_t^\varepsilon u$.
As shown in the next section, assumptions $(C1)$-$(C3)$ above ensure the convergence, in an appropriate sense,
of $\mathbb{T}^\varepsilon$ to $\mathbb{T}^0$ and of $\Phi_t^\varepsilon$ to $\Phi_t^0$.
For each $M>0$, we set
$$\mathcal{U}_M:=\left\lbrace u \ | \ u \in L^{\infty}((0,\infty);U), \mbox{\rm s.t. }
\|u\|_{ L^{\infty}((0,\infty);U)} \leqslant M \right\rbrace.$$
Denote $B(0,r)$, with $r>0$ fixed,  the closed ball in $X$, centered at the origin and of radius r.
We fix  $\psi\not \in {B}(0,r)$. For each $M>0$ and $\varepsilon\geqslant 0$, we define the
time optimal control problem:
\begin{equation}\label{time-optimal-control1.9}
(TP)^M_\varepsilon\qquad\qquad\tau_\varepsilon^*(M):= \displaystyle\min_{u \in \cal{U}_M} \left\lbrace t
\ |\ \mathbb{T}_t^\varepsilon \psi+ \Phi_t^\varepsilon u \in B(0,r)\right\rbrace.
\end{equation}
It is well known that the above optimal control problem admits at least one solution
$(\tau_\varepsilon^*(M) ,u_\varepsilon^*)$.
Our main purpose is to study the asymptotic behavior (when $\varepsilon \to 0$) of the solutions of $(TP)^M_\varepsilon$.
Our method to study this problem requires one more assumption:
\\ \ \\
(C4) For every $\varepsilon\geqslant 0$ and $\tau>0$, there exists $K(\varepsilon,\tau)>0$ such that
$$
K(\varepsilon,\tau)\int_0^\tau \|B^*\mathbb{T}_t^\varepsilon \psi\|_U\,  {\rm d} t\geqslant \|\mathbb{T}_\tau^\varepsilon\psi\|
\qquad(\psi\in X),
$$
where, for each $\varepsilon\geqslant 0$, the map $\tau\mapsto K(\varepsilon, \tau)$ is continuous on $(0, \infty)$.
\\ \ \\
Our main result states as follows:

\

\begin{theorem}\label{th_main}
With the above notation, let $(A_\varepsilon)_{\varepsilon\geqslant 0}$ and $B$ satisfy assumptions (C1)-(C4) above.
For every $M>0$ and $\varepsilon\geqslant 0$
we denote by $(\tau_\varepsilon^*(M),u_\varepsilon^*)$ the solution of the time optimal control problem $(TP)^M_\varepsilon$.
Then we have
\begin{equation}\label{time_conv}
\displaystyle\tau_\varepsilon^*(M)\to\tau_0^*(M),
\end{equation}
\vspace{-0.75cm}
\begin{equation}\label{con_conv}
u_\varepsilon^*\to u_0^*\quad \mbox{\rm strongly in}\quad L^2((0,\tau_0^*(M));U).
\end{equation}
We further assume that for $\varepsilon \geqslant 0$  we have
\begin{equation}\label{Lin_unique}
\eta\in X \quad{\rm and}\quad B^*\mathbb{T}_{t_0}^\varepsilon \eta =0 \qquad\hbox{\rm for some }\ t_0>0 \Longrightarrow \eta =0.
\end{equation} 
Then
for every $\delta \in (0,\tau_0^*(M))$ we have
\begin{equation}\label{con_conv_unif}
 \displaystyle u_\varepsilon^* \to u_0^*\quad\mbox{\rm strongly in} \quad L^\infty((0,\tau_0^*(M)-\delta);U).
\end{equation}
\end{theorem}
It is worth mentioning that similar convergence results have been obtained in Yu \cite{Yu} for a special case where the controlled systems are heat equations with small perturbations in the lower terms, see Example~\ref{heat_osc_simple}
in Section \ref{sec_final}. Note that the methods in \cite{Yu} do not seem applicable to the more general problem studied in our work.

The outline of the remaining part of this paper is as follows.
Section \ref{sec_approx} is devoted to study the convergence of abstract families
of analytic semigroups and some associated input to state maps when the operators
$A_\varepsilon$ converge to $A_0$ (in the sense of assumption (C3)).
Section \ref{sec_back_time} contains some necessary background on time and norm optimal control problems. In Section \ref{sec_relation}, we establish, in an abstract framework, a relation between  time and  norm optimal control problems. In Section \ref{sec_main_proof}, we provide the proof of our main theorem. In Section \ref{sec_final} we apply our abstract results to the homogenization problem stated above and we discuss other possible applications.
\section{Some approximation results for the semigroups}
\label{sec_approx}
\setcounter{equation}{0}

\

In this section, using the same notation as in the previous section, we study the behavior of the family of semigroups $(\mathbb{T}^\varepsilon)_{\varepsilon\geqslant 0}$ and of the input to state maps $(\Phi_\tau^\varepsilon)_{\varepsilon, \tau\geqslant 0}$,
introduced in \eqref{fieps}, when $\varepsilon \to 0$. These results can be seen as an abstract version of those on homogenization
of parabolic equations, as described, for instance, in \cite[Ch.2]{BLP}.

In the remaining part of this section, we assume that $(A_\varepsilon)_{\varepsilon \geqslant 0}$ satisfies
the assumptions (C1)-(C3) introduced in the previous section.
It is not difficult to check that there exist positive constants $K$ and $\alpha_0$ such that
\begin{equation}\label{unif_decay}
\|\mathbb{T}_{
\tau}^\varepsilon\|_{{\cal L}(X)}\leqslant e^{-\alpha_0 \tau} , \qquad \|\Phi_\tau^\varepsilon\|_{{\cal L}(L^\infty((0,\infty);U),X)}\leqslant K
\qquad(\tau, \;\varepsilon\geqslant 0)
\end{equation}

\begin{proposition}\label{cons_abs_1}
With the above notation, let $(A_\varepsilon)_{\varepsilon\geqslant 0}$ be a family of operators  satisfying
the assumption (C1)-(C3) in the previous section. Moreover, let $(\psi_\varepsilon)_{\varepsilon\geqslant 0}$ be a family of vectors in $X$.
For every $\varepsilon,t\geqslant 0$ we set  $\varphi_{\varepsilon}(t)=\mathbb{T}_t^\varepsilon \psi_\varepsilon$.
Then we have the following assertions:\\ \ \\
1. If $\displaystyle\psi_\varepsilon\to\psi_0$ weakly in $X$, then
$\displaystyle\varphi_{\varepsilon} \to\varphi_0$  weakly* in $L^\infty((0,\infty),X)$.\\

\noindent 2. If  $\displaystyle\psi_\varepsilon\to\psi_0$ strongly in $X$, then  for every $\tau>0$,
$ \displaystyle\varphi_\varepsilon\to\varphi_0 $ in $C([0,\tau];X)$.\\

\noindent3. If $\displaystyle\psi_\varepsilon\to\psi_0$ weakly in $X$ and  $(\tau_\varepsilon)_{\varepsilon \geqslant 0}$ is a family of positive numbers
with $\displaystyle\tau_\varepsilon\to\tau_0$,  then
$\displaystyle  \varphi_\varepsilon(\tau_\varepsilon) \to \varphi_0(\tau_0)$ strongly in $X$.
\end{proposition}

\

\begin{proof}  For each $\varepsilon,t \geqslant 0$,
we set $\widetilde \varphi_\varepsilon(t)=\mathbb{T}_t^\varepsilon \psi_0\;\; (t\geqslant 0)$.
We first show that
\begin{equation}\label{to prove 2.1.1}
\displaystyle \lim_{\varepsilon \to 0+} (\varphi_\varepsilon - \widetilde \varphi_\varepsilon)
= 0 \quad\mbox{\rm weakly* in }\quad L^\infty((0,\infty);X).
\end{equation}
Indeed, let $v\in L^1((0,\infty);X)$. By (C1)-(C3), we can apply  the Trotter-Kato theorem (see, for instance, \cite[Theorem 3.6.1,Proposition 3.6.2]{Cauchy}), to get  that,
$$
\displaystyle \lim_{\varepsilon\to 0}\left\|(\mathbb{T}_t^\varepsilon -\mathbb{T}_t^0) v(t)\right\| = 0
\qquad \qquad(t\geqslant 0\ {\rm a.e.}).
$$
By the first estimate in \eqref{unif_decay}, we have that $\|(\mathbb{T}_t^\varepsilon -\mathbb{T}_t^0) v(t)\|\leqslant 2\|v(t)\|$ for almost every $t\geqslant 0$, so that, by the dominated convergence theorem,
we obtain that the map $t\mapsto (\mathbb{T}_t^\varepsilon -\mathbb{T}_t^0) v(t)$ converges strongly to $0$ in $L^1((0,\infty);X)$ when $\varepsilon\to 0$.
This, together with the weak  convergence of $\psi_\varepsilon$ to $\psi_0$, gives \eqref{to prove 2.1.1}.

We next verify that
\begin{equation}\label{to prove 2.1.2}
\displaystyle \lim_{\varepsilon \to 0+} (\widetilde \varphi_\varepsilon - \varphi_0)
 = 0 \quad \mbox{\rm weakly* in }\quad L^\infty((0,\infty);X).
\end{equation}
Since $\{\widetilde \varphi_\varepsilon\}_{\varepsilon>0}$
 is bounded in $L^\infty((0,\infty);X)$,
it follows from  the Alaoglu theorem that,  up to the extraction of a subsequence,
\begin{equation}\label{Alao}
\displaystyle \lim_{\varepsilon\to 0+}\widetilde \varphi_\varepsilon = \widetilde \varphi_0 \qquad\hbox{\rm weakly* in }\quad L^\infty((0,\infty);X)
\end{equation}
for some $\widetilde \varphi_0\in L^\infty((0,\infty);X)$.
Hence,  still up to the extraction of a subsequence,
$$
\int_0^\infty {\rm e}^{-s t}\langle \widetilde \varphi_\varepsilon(t),\eta\rangle \, {\rm d}t
\to \int_0^\infty {\rm e}^{-s t}\langle \widetilde \varphi_0(t),\eta\rangle \, {\rm d}t \qquad(s>0,\ \eta\in X).
$$
Using a classical result (see, for instance, \cite[Proposition 2.3.1]{obsbook}),
we have
$$
\int_0^\infty {\rm e}^{-s t} \widetilde \varphi_\varepsilon(t)\, {\rm d}t =
(sI+A_\varepsilon)^{-1} \psi_0 \qquad(\varepsilon \geqslant 0,\ s>0),
$$
so that \ $\displaystyle \lim_{\varepsilon \to 0+}\langle (sI+A_\varepsilon)^{-1} \psi_0,\eta\rangle = \int_0^\infty
{\rm e}^{-s t}\langle \widetilde \varphi_0(t),\eta\rangle \, {\rm d}t$ for every
$s>0$ and $\eta\in X$.

On the other hand, assumption (C3), together with \cite[Proposition 3.6.2]{Cauchy}, implies that $\displaystyle\lim_{\varepsilon \to 0+}\| (sI+A_\varepsilon)^{-1} \psi_0-(sI+A_0)^{-1} \psi_0\|=0$ for every $s>0$.
Therefore, we have
$$
\langle (sI+A_0)^{-1} \psi_0,\eta\rangle = \int_0^\infty {\rm e}^{-s t}\langle \widetilde \varphi_0(t),\eta\rangle \,
{\rm d}t\qquad(s>0,\ \eta\in X).
$$
The last formula implies that
$$
 \int_0^\infty {\rm e}^{-s t}\langle  \varphi_0(t),\eta\rangle= \int_0^\infty {\rm e}^{-s t}\langle \widetilde
 \varphi_0(t),\eta\rangle \, {\rm d}t\qquad(s>0,\ \eta\in X),
$$
which, as well as the injectivity of the Laplace transform, yields that that $\widetilde \varphi_0=\varphi_0$.
This, along with \eqref{Alao}, yields  (\ref{to prove 2.1.2}). Thus, using \eqref{to prove 2.1.1} we obtain our first assertion.

To prove the second assertion, we notice that, since $(\mathbb{T}^\varepsilon)_{\varepsilon\geqslant 0}$
are contraction semigroups, we have
$$
\|\varphi_\varepsilon(t)-\varphi_0(t)\|\leqslant \|\psi_\varepsilon-\psi_0\|+\|\mathbb{T}_t^\varepsilon\psi_0-\mathbb{T}_t^0\psi_0\| \qquad(\varepsilon,\ t\geqslant 0).
$$
The first term of the above formula obviously tends to zero. To show that the same property holds for the second term,
 it suffices to apply the Trotter-Kato theorem.

To prove the third assertion in the proposition, we introduce $X_{-1}$, the dual space of $X_1$ with respect to the pivot
space $X$ and we denote by $\|\cdot\|_{-1}$ the norm of this space.
It is known (see, for instance, \cite[Section 2.10]{obsbook}) that, for every $\varepsilon\geqslant 0$,
the operator $A_\varepsilon$ can be extended to an operator in ${\cal L}(X,X_{-1})$ which
generates an analytic contraction semigroup on $X_{-1}$, extending $\mathbb{T^\varepsilon}$. The extended semigroup, still denoted by  $\mathbb{T}^\varepsilon$, satisfies
$\langle \mathbb{T}^\varepsilon_t \psi,\eta\rangle_{X_{-1},X_1} = \langle  \psi,\mathbb{T}^\varepsilon_t \eta\rangle$ for every $t\geqslant 0$, $\psi\in X_{-1}$ and $ \eta\in X_1$.
The compactness of the embedding $X_1\subset X$ clearly implies that of  the embedding $X\subset X_{-1}$,
so that, up to the extraction of a subsequence, $\displaystyle\lim_{\varepsilon\to 0+}\|\psi_\varepsilon-\psi\|_{-1}=0$.
According to classical properties of analytic contraction semigroups (see, for instance, \cite[Theorem 3.7.19]{Cauchy}) we have
\begin{eqnarray}\label{put_it_0}
\mathbb{T}_t^\varepsilon \in \mathcal{L}(X_{-1},X),\qquad \|\mathbb{T}_{t}^\varepsilon \|_{\mathcal{L}(X_{-1},X)} \leqslant t^{-1} \qquad(\varepsilon\geqslant 0,\ t>0).
\end{eqnarray}

Let us note that
\begin{equation}\label{hh3_new}
\| \mathbb{T}_{\tau_\varepsilon}^\varepsilon \psi_\varepsilon
- \mathbb{T}_{\tau_0}^0 \psi_0\| \leqslant \| \mathbb{T}_{\tau_\varepsilon}^\varepsilon \psi_\varepsilon
- \mathbb{T}_{\tau_\varepsilon}^\varepsilon \psi_0\| + \| \mathbb{T}_{\tau_\varepsilon}^\varepsilon \psi_0
- \mathbb{T}_{\tau_\varepsilon}^0 \psi_0\| + \| \mathbb{T}_{\tau_\varepsilon}^0 \psi_0 - \mathbb{T}_{\tau_0}^0 \psi_0\|.
\end{equation}
For the first term on the right hand side of \eqref{hh3_new}, we use \eqref{put_it_0} to get
$$
 \| \mathbb{T}_{\tau_\varepsilon}^\varepsilon \psi_\varepsilon - \mathbb{T}_{\tau_\varepsilon}^\varepsilon \psi_0\| \leqslant  \|  \mathbb{T}_{\tau_\varepsilon}^\varepsilon \|_{\mathcal{L}(X_{-1},X)} \|\psi_\varepsilon-\psi_0\|_{-1} \leqslant \frac{\|\psi_\varepsilon-\psi_0\|_{-1}} {\tau_\varepsilon},
$$
so that this term clearly converges to $0$.
To see that the second term on the right hand side of \eqref{hh3_new} converge to $0$,
it suffices to apply the second assertion of this proposition. Finally, the last term on
the right hand side of \eqref{hh3_new} converge to $0$, because of the continuity of $t\mapsto \mathbb{T}_t$. Thus, the third assertion holds, which ends the proof.
\qquad \end{proof}

\

\begin{proposition}\label{cons_abs_1_bis}
With the notation and assumptions in Proposition \ref{cons_abs_1}, let $\tau>0$ and
let $(u_\varepsilon)_{\varepsilon \geqslant 0}$
be a family of functions in $L^\infty((0,\infty);U)$.
Set, for each $\varepsilon,t\geqslant 0$,
$w_{\varepsilon}(t)= \Phi_t^\varepsilon u_\varepsilon$,
where $\Phi_t^\varepsilon$ is defined by \eqref{fieps}.
Then the following assertions hold:

\noindent 1. If \ $\displaystyle u_\varepsilon \to u_0$\  weakly$^*$ in $ L^\infty((0,\infty);U)$, \ then \
$\displaystyle w_\varepsilon \to w_0$  weakly$^*$ in $L^\infty((0,\infty);X)$.\\

\noindent 2. If $\displaystyle u_\varepsilon \to u_0$\ weakly$^*$ in $ L^\infty((0,\infty);U)$, then \
$\displaystyle w_\varepsilon(\tau) \to w_0(\tau)$ weakly in $X$.\\

\noindent 3. If  $u_\varepsilon\to u_0$ strongly in $L^\infty((0,\infty);U)$, then
$\displaystyle w_\varepsilon \to w_0$ in $C([0,\tau];X)$.
\end{proposition}

\

\begin{proof}
For each $\varepsilon,t\geqslant 0$ we set $\widetilde w_\varepsilon (t)=
\Phi_t^\varepsilon u_0\ \ (t\geqslant 0)$.
We  first claim that
\begin{equation}\label{faible_prel_bis}
\displaystyle \lim_{\varepsilon \to 0+}(\widetilde w_\varepsilon-w_\varepsilon)=0 \qquad\hbox{\rm weakly* in }\quad\ L^\infty((0,\infty);X).
\end{equation}
Indeed, for each $v\in L^1((0,\infty);U)$ and $\varepsilon\geqslant 0$, we deduce from Fubini's theorem that
\begin{align}\label{insiruire}
 \int_0^\infty \left\langle\widetilde w_\varepsilon(t)- w_\varepsilon(t),v(t)  \right\rangle\, {\rm d}t &= \int_0^\infty \left\langle \int_0^t \mathbb{T}_{t-\sigma}^\varepsilon B \left( u_\varepsilon(\sigma)-u_0(\sigma) \right)\, {\rm d}\sigma, v(t) \right\rangle\, {\rm d}t \notag
\\
&\ = \int_0^\infty \int_\sigma^\infty \left\langle \mathbb{T}_{t-\sigma}^\varepsilon B \left( u_\varepsilon(\sigma)-u_0(\sigma) \right) , v(t) \right\rangle\, {\rm d}t\, {\rm d}\sigma \notag
\\
& = \int_0^\infty  \left\langle   u_\varepsilon(\sigma)-u_0(\sigma) , \int_\sigma^\infty B^* \mathbb{T}_{t-\sigma}^\varepsilon  v(t)\, {\rm d}t\right\rangle\,  {\rm d}\sigma.
\end{align}
For each $\varepsilon, \sigma \geqslant 0$, we set $g_\varepsilon(\sigma): =\int_\sigma^\infty B^* \mathbb{T}_{t-\sigma}^\varepsilon  v(t)\, {\rm d}t$.
By the assumptions (C1)-(C3), we can apply the dominated convergence theorem to get that
$$
\displaystyle \lim_{\varepsilon \to 0+} g_{\varepsilon}(\sigma) =g_0(\sigma) =\int_\sigma^\infty B^* \mathbb{T}_{t-\sigma}^0
v(t)\, {\rm d}t \qquad(\sigma\geqslant 0).
$$
Moreover, using the first estimate in \eqref{unif_decay}, we have that for each $\sigma\geqslant 0$
$$
\| g_\varepsilon(\sigma)-g_0(\sigma) \|  \leqslant 2
\|B^*\|_{{\cal L}(X,U)} \, {\rm e}^{\alpha_0 \sigma}\,  \int_\sigma^\infty {\rm e}^{-\alpha_0 t} \|v(t) \| \, {\rm d}t \leqslant 2 \|B^*\|_{{\cal L}(X,U)} \ \|v\|_{L^1((0,\infty);U)}.
$$
From these, we can apply  the dominated convergence theorem again to obtain  that
$g_\varepsilon \to g_0$ strongly in $L^1((0,\infty);U)$.
This, combined with \eqref{insiruire}, implies  \eqref{faible_prel_bis}.

We next show
\begin{equation}\label{wu2.10}
\widetilde w_\varepsilon\to w_0 \qquad\hbox{\rm weakly* in }\quad L^\infty((0,\infty);X).
\end{equation}
By the second estimate in \eqref{unif_decay}, we have that $(\widetilde w_\varepsilon)_{\varepsilon>0}$ is bounded in $L^\infty((0,\infty);X)$.
 Hence,
up to the extraction of a subsequence, we have that
\begin{equation}\label{conv_prel_bis}
\displaystyle\lim_{\varepsilon \to 0+}\widetilde w_\varepsilon = \widetilde w_0 \qquad\hbox{\rm weakly* in }\quad L^\infty((0,\infty);X),
\end{equation}
for some $\widetilde w_0 \in L^\infty((0,\infty);X)$.
The above convergence implies that
\begin{equation}\label{wu2.12}
\displaystyle\lim_{\varepsilon \to 0+}\int_0^\infty {\rm e}^{-s t}\langle \widetilde w_\varepsilon(t),\eta\rangle \, {\rm d}t
 = \int_0^\infty {\rm e}^{-s t}\langle \widetilde w_0(t),\eta\rangle \, {\rm d}t \qquad(s>0,\ \eta\in X).
\end{equation}
Two observations are given in order. Firstly, by a classical result (see, for instance, \cite[Remark 4.1.9]{obsbook}),
we have
$$
\int_0^\infty {\rm e}^{-s t} \widetilde w_\varepsilon(t)\, {\rm d}t =
(sI+A_\varepsilon)^{-1} B \widehat u_0(s) \qquad(s>0),
$$
where $\widehat u_0$ stands for the Laplace transform of $u_0$. Therefore, it follows that
\begin{equation}\label{wu2.13}
\displaystyle\lim_{\varepsilon \to 0+}\left\langle (sI+A_\varepsilon)^{-1} B\widehat u_0(s),\eta\right\rangle
= \int_0^\infty {\rm e}^{-s t}\langle \widetilde w_0(t),\eta\rangle \, {\rm d}t\qquad(s>0,\ \eta\in X).
\end{equation}
Secondly,  assumption (C3), together with \cite[Proposition 3.6.2]{Cauchy}, implies that
\begin{equation}\label{wu2.14}
\displaystyle\lim_{\varepsilon\to 0+}\left\| (sI+A_\varepsilon)^{-1} B\widehat u_0(s)-(sI+A_0)^{-1} B\widehat u_0(s)\right\|=0 \qquad(s>0).
\end{equation}
From (\ref{wu2.12})-(\ref{wu2.14}), it follows that
$$
\left\langle (sI+A_0)^{-1} B\widehat u_0(s),\eta\right\rangle = \int_0^\infty {\rm e}^{-s t}\langle \widetilde w_0(t),\eta\rangle \, {\rm d}t\qquad(s>0,\ \eta\in X).
$$
The last formula implies that
$$
 \int_0^\infty {\rm e}^{-s t}\langle  w_0(t),\eta\rangle= \int_0^\infty {\rm e}^{-s t}\langle \widetilde w_0(t),\eta\rangle \,
 {\rm d}t\qquad(s>0,\ \eta\in X).
$$
This, along with  the injectivity of the Laplace transform, yields that  $\widetilde w_0=w_0$, which leads to  \eqref{wu2.10}.
 Combining this and \eqref{faible_prel_bis}, we obtain our first assertion.

The proof of the second assertion follows from the proof of (\ref{faible_prel_bis}) where we take $v(t)\equiv \eta$
 for arbitrarily fixed $\eta\in X$.

To prove the third assertion of the proposition, it suffices to note that
\begin{equation} \label{to prove}
\|w_\varepsilon(t)-w_0(t)\| \leqslant \|\Phi_t^\varepsilon(u_\varepsilon-u_0) \| + \|(\Phi_t^\varepsilon-\Phi_t^0 ) u_0 \|.
\end{equation}
The first term on the right hand side of \eqref{to prove} clearly converges to $0$ because of the second estimate in \eqref{unif_decay}.
The second term on the right hand side of \eqref{to prove} clearly converges to $0$ according to the Trotter-Kato theorem
and the dominated convergence theorem. This ends the proof.\qquad \end{proof}

From the two above propositions, we clearly have the following consequences:

\

\begin{corollary}\label{clasica_1}
With the notation and assumptions in Proposition \ref{cons_abs_1}, let $(\psi_\varepsilon)_{\varepsilon \geqslant 0}$ and $(u_\varepsilon)_{\varepsilon \geqslant 0}$ be two families in  $X$ and $L^\infty((0,\infty);U)$, respectively.
For each $\varepsilon,t\geqslant 0$ we set
$z_\varepsilon(t)=\mathbb{T}_{t}^\varepsilon \psi_\varepsilon + \Phi_t^\varepsilon u_\varepsilon$. Then the following assertions hold:

\noindent 1. If $\displaystyle \psi_\varepsilon\to \psi_0$ weakly in $X$ and
 $\displaystyle u_\varepsilon\to u_0$ weakly* in  $L^\infty((0,\infty);X)$ then
$\displaystyle z_\varepsilon\to z_0$ weakly* in \  $L^\infty((0,\tau);X)$ and, for each $\tau\geqslant 0$, we have $\displaystyle  z_\varepsilon(\tau) \to z_0(\tau)$ weakly in $X$.\\

\noindent 2. If $\displaystyle \psi_\varepsilon\to \psi_0$ strongly in $X$ and
 $\displaystyle u_\varepsilon\to u_0$ strongly in  $L^\infty((0,\infty);U)$, then for each $\tau\geqslant 0$, we have
 $ \displaystyle  z_\varepsilon \to z_0$ in $C([0,\tau];X)$.
\end{corollary}


\section{Some background  on time and norm optimal control problems}
\label{sec_back_time}
\setcounter{equation}{0}

\

In this section we recall some basic results on time and norm optimal controls for linear,
time invariant infinite dimensional systems. We give, in particular, versions of Pontryagin's maximum principle for time and norm optimal control problems in the case when the target is a closed ball in a Hilbert space. These kind of results are widely discussed in the literature, in a very general context, for  linear and nonlinear equations with various target sets. We refer, for instance, to Fattorini \cite{fattorini_old_book,Fattorini_old} and \cite{liyong}. However, for the sake of completeness, we give in Appendix a short proof adapted to the particular situation considered in this work. Moreover, we state the analogous results for the associated norm optimal control problem.

Throughout this section,  $X$ and $U$ are real Hilbert spaces, $\mathbb{T}=(\mathbb{T})_{t\geqslant 0}$ is a $C^0$
semigroup of linear operators on $X$, with the  generator $-A$, and $B\in \mathcal{L}(U,X)$ is a bounded control operator.
For every $t \geqslant 0,\ \psi\in X$ and $u\in L^{\infty}((0,\infty);U)$, we set
\begin{equation}\label{used_notation}
z(t;\psi,u):= \mathbb{T}_t\psi+\Phi_t u\quad \mbox{\rm with }\quad
\Phi_t u := \int_0^t \mathbb{T}_{t-\sigma}  B u(\sigma)\, {\rm d}\sigma.
\end{equation}
It is well-known that $z$ is the state trajectory of the system $(\Sigma)$, described by
$$
\dot z(t)+A z(t)=Bu(t) \qquad (t\geqslant 0),\qquad
\qquad z(0)=\psi\in X.  \leqno{(\Sigma)}
$$
Fix $r>0$ and  $\psi\not \in {B}(0,r)$. Consider the following time optimal  control problems:
\begin{equation}\label{minimal time func}
(TP)^M\qquad \qquad \tau^*(M):= \displaystyle\min_{u \in \cal{U}_M} \left\lbrace t \ |\ z(t;\psi,u)\in B(0,r)\right\rbrace\ \ (M>0),
\end{equation}
 where $\mathcal{U}_M:=\left\lbrace u \ | \ u \in L^{\infty}((0,\infty);U)\quad \mbox{\rm s.t. }
\|u\|_{ L^{\infty}((0,\infty);U)} \leqslant M \right\rbrace$.

\

\begin{theorem}\label{bb}
With the above notation and assumptions, let $M>0$ and  suppose that $\psi \not \in B(0,r)$.
Let $(\tau^*(M),u^*)$ be a solution of the time optimal control problem $(TP)^M$. Then
 $u^*$ satisfies the Pontryagin's maximum principle, i.e., there exists $ \eta  \in X$, with $ \eta \not= 0$, such that for almost every  $t\in (0,\tau^*(M))$, we have
\begin{equation}\label{PM}
\ \left\langle  u^*(t),B^* \mathbb{T}_{\tau^*(M)-t}^* \eta \right\rangle_U
= \displaystyle\max_{v\in U,\; \|v\|_U \leqslant M}\ \left\langle  v,B^* \mathbb{T}_{\tau^*(M)-t}^* \eta \right\rangle_U.
\end{equation}
Moreover, we have the following  transversality condition: $\eta= -z(\tau^*(M);\psi,u^*)$.
\end{theorem}

\

The proof of Theorem \ref{bb} is given in Appendix for the sake of completeness of the paper.
From Theorem \ref{bb}, we clearly have the following consequence:

\

\begin{corollary}\label{bb_cor}
With the notation and assumptions in Theorem \ref{bb}, suppose
\begin{equation}\label{uni-cont-3.5}
\eta\in X\;\;\mbox{\rm and}\;\;B^*\mathbb{T}_t^* \eta=0\;\;\mbox{\rm for all}\;\;t\;\;\mbox{\rm in a set of positive measure}\;
\Longrightarrow \;\eta=0.
\end{equation}
  Then the following assertions hold:\\
\noindent 1. Any  optimal control $u^*$ to the problem  $(TP)^M$  has the  bang-bang property, i.e., it satisfies that
 $\|u^*(t)\|=M$  for almost every $t\in (0, \tau^*(M))$.\\
2. The optimal control to the problem $(TP)^M$, if exists, is unique.
\end{corollary}

\

Our study on time optimal control problems is related to the following norm optimal control problems, defined for every $\tau >0$ by
\begin{equation}\label{minimal norm func}
(NP)^\tau \quad  N^*(\tau) := \displaystyle\min_{f \in L^{\infty}((0,\tau);U)} \left\lbrace \|f\|_{L^{\infty}((0,\tau);U)}\ | \
 \; z(\tau;\psi,f)\in B(0,r)\right\rbrace.
\end{equation}

\noindent The Pontryagin maximum principle for problems  $((NP)^\tau)_{\tau >0}$ states as follows:

\

\begin{theorem}\label{bb_n}
With the notation and assumptions in Theorem \ref{bb}, let $\tau>0$ and
 assume that $\|\mathbb{T}_\tau\psi\|>r$. Let $\hat f$ be a norm optimal control to problem $(NP)^\tau$.
Then, it satisfies the Pontryagin maximum principle, i.e., there exists $ \eta  \in X$ with $\eta \not= 0$, such that for almost every $t\in (0,\tau)$, we have:
\begin{equation}\label{PM_n}
\ \left\langle \hat f(t),B^* \mathbb{T}^*_{\tau-t} \eta \right\rangle_U = \displaystyle\max_{v\in U,\;
 \|v\|_U \leqslant N^*(\tau)}\ \left\langle  v,B^* \mathbb{T}^*_{\tau-t} \eta \right\rangle_U.
\end{equation}
Moreover, we have the transversality condition: $\eta= -z(\tau;\psi, \hat f)$.
\end{theorem}

For the sake of completeness, a short proof of Theorem \ref{bb_n} will be  given in Appendix  of the paper. From Theorem \ref{bb_n},
 it easily follows

\

\begin{corollary}\label{bb_n_cor}
With the same notation and assumptions in Theorem \ref{bb_n}, suppose that \eqref{uni-cont-3.5} holds.
Then the following statements hold: \\
\noindent 1. Any optimal control $\hat f$ to the problem $(NP)^\tau$ has the bang-bang property, i.e., it  satisfies that $\|\hat f(t)\|=N^*(\tau) $ for a.e. $t\in (0, \tau)$.\\
\noindent 2. The optimal control to the problem $(NP)^\tau$, if exists, is unique.
\end{corollary}

To further study the norm optimal control problem $(NP)^\tau$, with $\tau>0$, it is useful to introduce an auxiliary
 minimization problem. This has been remarked in the case of the heat equations
in Fabre {\it et al} \cite{Fabre}. More precisely, we introduce  the following minimization problem:
\begin{equation}\label{minimization func}
 (JP)^\tau  \qquad \qquad V^*(\tau) :=\displaystyle \min_{\eta\in X} \left\lbrace J^\tau(\eta)\right\rbrace,
\end{equation}
where
$$
J^\tau(\eta) = \frac{1}{2} \left(\int_0^\tau \|B^*\mathbb{T}^*_{\tau-t} \eta\|_U \, {\rm d}t \right)^2
+ \left\langle \psi,\mathbb{T}^*_\tau \eta\right\rangle+r \|\eta\| \qquad\qquad(\eta\in X).
$$
We first give in the following lemma some properties concerning the problem $(JP)^\tau$.

\

\begin{lemma}\label{JT}
 With the same notation and assumptions in Theorem \ref{bb_n}, let $\tau>0$. Moreover, assume that the pair $(A^*,B^*)$ is approximatively observable in any time.
Then the following assertions hold:\\
\noindent 1. The functional $J^\tau$ is  continuous, strictly convex and coercive on $X$.\\
\noindent 2. The functional $J^\tau$ has a unique minimizer over $X$.\\
\noindent 3. $\mathbb{T}_\tau \psi \not \in B(0,r)$  if and only if zero is not the minimizer of $J^\tau$.\\
\end{lemma}

\begin{proof}
The continuity and the strict convexity of the functional  $J^\tau$ are obvious. The coercivity of $J^\tau$  can be verified by the exactly same way used in \cite[Ch.4.2]{Zuazua}, using in an essential way the approximate observability of $(A^*,B^*)$.  Moreover, since a stronger version of this coercive property will be proved in Proposition \ref{coercivity} below, we omit here the proof.

To prove the second assertion, we first check that
\begin{equation}\label{put_a_number}
\mathbb{T}_{\tau} \psi \not \in B(0,r)\; \Rightarrow \; \mbox{\rm 0 is not the minimizer of}\; J^{\tau}.
\end{equation}
Indeed, assume by contradiction that $\mathbb{T}_\tau \psi \not \in B(0,r) $, but $0$  were the minimizer of $J^\tau$. Then we would have that for every $\eta$ in $X$, $J^\tau(\lambda\eta)/\lambda\geqslant J^\tau(0)/\lambda=0$. Passing to the limit for $\lambda\rightarrow 0$ in the later, we have
 that $|\left\langle \mathbb{T}_\tau \psi, \eta \right\rangle | \leqslant r \|\eta \|
$ for all $\eta \in X$. Thus, $\mathbb{T}_\tau \psi \in B(0,r)$, which leads to a contradiction. Hence \eqref{put_a_number} holds.

Conversely, we show that if $\mathbb{T}_\tau \psi \in B(0,r)$, then $0$ is the minimizer of $J^\tau$.
For this purpose, we note that
if $\mathbb{T}_\tau \psi \in B(0,r)$, then $ |\left\langle \psi,\mathbb{T}^*_\tau \eta \right\rangle |\leqslant r \|\eta \|$ for all $\eta \in X$. Hence, for any $\;\eta \in X$, it is clear that $
J^\tau(\eta) \geqslant \left\langle \psi, \mathbb{T}^*_{\tau}\eta\right\rangle +r \|\eta \| \geqslant 0 =J^\tau(0)$.
This implies that $0$ is the minimizer of $J^\tau$ and ends the proof. \qquad \end{proof}

The following  proposition gives an explicit formula of the norm optimal control in terms of the minimizer of $J^\tau$.

\

\begin{proposition}\label{JT and NT}
With the same notation and assumptions in Theorem \ref{bb_n}, suppose that $\mathbb{T}_\tau \psi \not \in B(0,r)$ and that property \eqref{uni-cont-3.5} holds. Let $\widehat \eta$ be the minimizer of $J^\tau$. Then we have: \\
\noindent 1. The function $\widehat{f}$  defined by
\begin{eqnarray}\label{bb control}
\widehat{f}(t) = \left( \int_0^\tau \|B^*\mathbb{T}^*_{\tau-s}\widehat \eta \|_{U}\, {\rm d}s  \right) \frac{B^*\mathbb{T}^*_{\tau-t}\widehat \eta}{\| B^*\mathbb{T}^*_{\tau-t}\widehat \eta\|_{U}} \qquad\qquad(t\in(0,\tau)\ {\rm a.e.}),
\end{eqnarray}
is the norm optimal control to the problem $(NP)^\tau$.\\
\noindent 2. $N^*(\tau)= \displaystyle \int_0^\tau \| B^*\mathbb{T}^*_{\tau-t}\widehat \eta\|_{U}\, {\rm d}t$, where $N^*(\tau)$ is defined by \eqref{minimal norm func}.\\
\noindent 3. $V^*(\tau)=\displaystyle -\frac{1}{2} {N^*(\tau)}^2$, where $V^*(\tau)$ is defined by \eqref{minimization func}.
\end{proposition}

\

\begin{proof}
Since the uniqueness of the optimal control to the problem $(NP)^\tau$ has been proved in Corollary \ref{bb_n_cor}, we only need to show that $\hat f$, defined by (\ref{bb control}), is an optimal control to the problem $(NP)^\tau$.
For this purpose, we first use  Lemma \ref{JT} to get  that $\widehat \eta\neq 0$. Then it follows from the assumption \eqref{uni-cont-3.5} that $B^*\mathbb{T}^*_{\tau-t} \widehat \eta\neq 0$ for every $t\in [0,\tau)$. Consequently, $\widehat{f}$ is well defined.

We next  prove that $\widehat{f}$ is an admissible control to Problem  $(NP)^\tau$, which means that  $\|z(\tau;\psi,\widehat f)\|\leqslant r$. To this end, we observe that the Euler-Lagrange equation, associated with
 the minimizer $\widehat \eta$ of $J^\tau$, reads:
\begin{eqnarray}
\int_0^\tau \left\langle \widehat{f},B^*\mathbb{T}^*_{\tau-t}\eta\right\rangle_U\, {\rm d}t  +\left\langle \psi,\mathbb{T}^*_\tau\eta\right\rangle + r \left\langle {\widehat \eta}/{\|\widehat \eta\|},\eta\right\rangle=0 \qquad(\eta \in X). \label{EL}
\end{eqnarray}
This, along with \eqref{used_notation}, yields that $\|z(\tau;\psi, \widehat{f})\| \leqslant r$.

We next show the optimality of $\widehat{f}$ to the problem $(NP)^\tau$. First of all, we denote
$\varphi(t)=B^*\mathbb{T}^*_{\tau-t}\widehat\eta$ for all $t\in [0,\tau]$. Then for each  $v\in L^\infty((0,\tau); U)$, we clearly have
\begin{equation*}
\int_0^\tau \langle\widehat{f}(t), \varphi(t) \rangle_U\, {\rm d}t - \langle z(\tau;\psi,\widehat{f}), \widehat \eta\, \rangle =
\int_0^\tau \langle v(t), \varphi(t)\rangle_U\, {\rm d}t
 - \langle z(\tau;\psi,v), \widehat \eta\, \rangle.
\end{equation*}
This, together with (\ref{EL}) with $\eta=\widehat\eta$, yields that when  $v\in L^\infty((0,\tau); U)$, we have
$$
\|\widehat{f} \|_{L^\infty((0,\tau);U)}^2=\int_0^\tau \left\langle v(t),  \varphi(t)  \right\rangle_U\, {\rm d}t - \left\langle z(\tau;\psi,v), \widehat \eta\, \right\rangle -r \|\widehat \eta\|.
$$
Since $\|\widehat f(t)\|_U=\|\varphi(t)\|_U$ for every $t\in [0,\tau)$, the above equality implies that when $v$ is admissible to the problem $(NP)^\tau$, we have
$$
\|\widehat{f} \|_{L^\infty((0,\tau);U)}\leqslant\|v \|_{L^\infty((0,\tau);U)}.
$$
Hence,  $\widehat{f}$ is the optimal control to $(NP)^\tau$, which ends the proof of our second assertion.

The last assertion follows from \eqref{EL}, with  $\eta=\widehat \eta$, and from \eqref{minimization func}.
 \qquad \end{proof}

\section{On the relation of time and norm optimal control problems}
\label{sec_relation}\setcounter{equation}{0}

\

In this section, we present some relationships between time and norm optimal control problems.
We use the same notation and assumptions on $X,\ U,\ A,\ \mathbb{T},\ (\Phi_t)_{t\geqslant 0},\ B$ as in the previous section. In particular, for every $t \geqslant 0,\ \psi\in X$ and $u\in L^{\infty}((0,\infty);U)$, we set
$$
z(t;\psi,u):= \mathbb{T}_t\psi+\Phi_t u\quad \mbox{\rm with }\quad
\Phi_t u := \int_0^t \mathbb{T}_{t-\sigma}  B u(\sigma)\, {\rm d}\sigma.
$$
Moreover, we assume that $-A:{\cal D}(A)\to X$ generates a contraction semigroup  $\mathbb{T}$. For fixed $r>0$ and $\psi\in X$ we denote
\begin{equation}\label{TAU0}
\hat \tau=\begin{cases}
\displaystyle\min_{t>0}  \{t \ \ | \ \ \|\mathbb{T}_t\psi\|\leqslant r\} & \ \hbox{\rm if }\ \  \{t \ \ | \ \ \|\mathbb{T}_t\psi\|\leqslant r\}\neq \emptyset\\
+\infty & \ \hbox{\rm if } \ \ \{t \ \ | \ \ \|\mathbb{T}_t\psi\|\leqslant r\} = \emptyset.
\end{cases}
\end{equation}
We first present some properties of functions $\tau^*(\cdot)$ and $N^*(\cdot)$ defined by
 (\ref{minimal time func}) and (\ref{minimal norm func}), respectively.

\

\begin{proposition}\label{NT}  Let $r>0$ and $\psi \in X\setminus B(0,r)$.
Assume that for every $\tau>0$, we have
\begin{equation}\label{OBSL1}
K(\tau)\int_0^\tau \|B^*\mathbb{T}^*_t \eta\|_U \, {\rm d}t\geqslant \|\mathbb{T}^*_\tau\eta\| \qquad
\qquad(\eta\in X),
\end{equation}
where $K(\cdot): (0,\infty)\rightarrow (0, \infty)$ is continuous.
Let $\widehat{\tau}$ be the quantity defined in \eqref{TAU0}. Then the following statements hold:\\

\noindent 1. The map $t \mapsto N^*(t)$ is decreasing and continuous  on $(0,\widehat{\tau})$.\\
\noindent 2.  $\displaystyle \lim_{t \to 0+} N^*(t) = +\infty$ and $\displaystyle \lim_{t \to \widehat{\tau}} N^*(t) = 0$.\\
\noindent 3.  The functions $N^*(\cdot)$ (restricted to $(0,\widehat{\tau})$) and $\tau^*(\cdot)$  are inverse to each other, i.e.,
\begin{eqnarray}
N^*(\tau^*(M))=M \qquad (M>0)\;\; \mbox{\rm and}\;\;  \tau^*(N^*(\tau))=\tau \qquad(\tau\in (0,\widehat{\tau})).
\end{eqnarray}

\end{proposition}

\begin{proof}
We first show that the considered map is decreasing. Suppose, by contradiction, that there exist  $\tau_1, \tau_2\in (0, \widehat{\tau})$, with   $\tau_1<\tau_2$, such that
\begin{equation}\label{assume_con}
N^*(\tau_2) \geqslant N^*(\tau_1)>0.
\end{equation}
On one hand, let $f^*$ be the extension of the norm optimal control of $(NP)^{\tau_1}$ over $(0,\tau_2)$ with zero value over $(\tau_1, \tau_2)$. Then, it follows that
\begin{equation}\label{yuanyuan4.42}
\|f^*\|_{L^\infty((0, \tau_2); U)}=N^*(\tau_1)\;\;\mbox{\rm and}\;\; z(\tau_1;\psi,f^*) \in B(0,r).
\end{equation}
On the other hand, due to a classical duality argument (see, for instance, Micu et al \cite[Proposition 2.6]{MRT}), \eqref{OBSL1} implies that there is a control $v \in L^\infty((0,\tau_2),U)$, supported in the interval
$[\tau_1,\tau_2]$, such that
\begin{equation}\label{yuanyuan4.43}
 z(\tau_2;\psi,v) =0\;\;\mbox{\rm and}\;\;
\|v\|_{L^\infty((0, {\tau_2});U)}\leqslant C \|\psi\|,
\end{equation}
where $C>0$ depends on $r$, $\tau_1$, $A$ and  $(\tau_2-\tau_1)$.
Choose $\lambda_0 \in (0,1)$ such that
\begin{equation}\label{marius3.20}
 C \lambda_0 \|\psi\| < N^*(\tau_1).
 \end{equation}

We now define a new control $\widetilde f$ over $(0,\tau_2)$ by
\begin{eqnarray}\label{marius3.21}
	\widetilde f= \chi_{[0,\tau_1)} (1-\lambda_0) u_1^* + \chi_{[\tau_1,\tau_2]} \lambda_0 v.
\end{eqnarray}	
Since $\mathbb{T}_t$ is a contraction semigroup, it follows from (\ref{yuanyuan4.42}) and (\ref{yuanyuan4.43}) that
  $\widetilde f$ is an admissible control to problem $(NP)^{\tau_2}$, which implies that
$ \|\widetilde f\|_{L^\infty((0,\tau_2);U)}\geqslant N^*(\tau_2)$.
Meanwhile, it follows from (\ref{marius3.21}), (\ref{yuanyuan4.42})  and (\ref{marius3.20}) that
$ \|\widetilde f\|_{L^\infty((0,\tau_2);U)}< N^*(\tau_1)$.
The above two inequalities  contradict the assumption \eqref{assume_con}.
Consequently $\tau \mapsto N^*(\tau)$ is decreasing on $(0,\widehat{\tau})$.

We next show that $\tau \mapsto N^*(\tau)$ is right continuous. To this end,
let $(\tau_n)_{n>0} \subset (0,\widehat{\tau})$ be an arbitrary decreasing sequence with $\displaystyle \lim_{n\to \infty} \tau_n =\tau$ and let $f^*_n$ be the extension of the norm optimal control of $(NP)^{\tau_n}$ over $(0,\infty)$ with zero value over $(\tau_n, \infty)$ for each $n\in \mathbb{N}$.  Then it is clear that
$$
z(\tau_n;,\psi,f_n^*)\in B(0,r)\;\;\mbox{\rm and}\;\; \| f_n^*\|_{L^{\infty}((0,\infty);U)} = N^*(\tau_n) \leqslant N^*(\tau),
$$
Since $(f_n^*)_{n\in\mathbb{N}}$ is bounded in $L^\infty((0,\infty);U)$, there exists a subsequence, denoted in the same way, and a control $g\in L^{\infty}((0,\infty);U)$ such that
$f_n^* \to g$ weakly* in $L^{\infty}((0,\infty);U)$ and $\;z(\tau_n;\psi, f_n^*) \to z(\tau;\psi, g)$  weakly in $X$. Then, it follows that
$z(\tau,\psi,g) \in B(0,r),$
which implies that $g$, when restricted over $(0,\tau)$,  is an admissible control to $(NP)^{\tau}$.
Thus, we have
$$
N^*(\tau) \leqslant \| g\|_{L^{\infty}((0,\tau);U)} \leqslant \displaystyle \liminf_{n \to \infty} \|f_n^*\|_{L^{\infty}((0,\tau);U)} = \displaystyle \lim_{n\to \infty} N^*(\tau_n) \leqslant N^*(\tau).
$$
Hence,  $\tau \mapsto N^*(\tau)$ is right continuous.

We now prove that the map $\tau \mapsto N^*(\tau)$ is left continuous.
Assume, by contradiction, that there did exist a $\delta>0$ such that
\begin{eqnarray}
N^*(\tau_n) >N^*(\tau)+\delta,\quad\mbox{\rm when}\quad n\in \mathbb{N},\label{lc1}
\end{eqnarray}
for some $(\tau_n)_{n>0} \subset (0,\widehat{\tau})$ with $\displaystyle \lim_{n\to \infty} \tau_n =\tau$.
To obtain a contradiction, it suffices to show the existence of $n_0\in \mathbb{N}$ and of a control $f_{n_0}$ such that
\begin{equation}\label{marius23}
\|{f}_{n_0}\|_{L^\infty((0,\infty);U))} < N^*(\tau_{n_0})\;\;\mbox{and}\;\;z(\tau_{n_0}; \psi, {f}_{n_0})\in B(0,r).
\end{equation}

We now prove the existence of such $n_0$ and $f_{n_0}$. Let $f^*$ be the extension of the norm optimal control of $(NP)^{\tau}$ over $(0,\infty)$ with zero value over $(\tau,\infty)$. Then,
\begin{equation}\label{yuanyuan4.49}
z(\tau; \psi, f^*)\in B(0,r)\;\;\mbox{\rm and}\;\; \|f^*\|_{L^\infty((0,\infty); U)}=N^*(\tau).
\end{equation}
By the continuity of $\tau \mapsto z(\tau;\psi, f^*)$, we have that given $\rho>0$, there exists a
natural number $n_1(\rho)$ such that for every $n > n_1(\rho)$, we have $ \| z(\tau_n;\psi, f^*)-z(\tau;\psi,f^*)\| \leqslant\rho$. This, along with (\ref{yuanyuan4.49}), yields that  for each $\lambda \in (0,1)$,
\begin{eqnarray}\label{lc2}
\| z(\tau_n;\lambda \psi, \lambda f^*)\| \leqslant\lambda (\rho+r) \quad \mbox{\rm when}\quad n>n_1(\rho).
\end{eqnarray}
Meanwhile, by the assumption \eqref{OBSL1}, we know that for each $n\in \mathbb{N}$, there exists  a control $v_{n}$, supported over
$(0,\tau_n)$, such that
\begin{eqnarray}\label{lc3}
\|v_n\|_{L^\infty((0,\infty);U))} \leqslant C \|\psi\|\;\;\mbox{\rm and}\;\; z(\tau_n;\psi,v_n)=0,
\end{eqnarray}
where $C$ is independent of $n$.
Define, for each $\lambda\in(0,1)$ and $n\in \mathbb{N}$,
\begin{eqnarray}\label{1-lc4}
f_{\lambda,n}:=\lambda f^* +(1-\lambda)v_{n}\;\;\mbox{\rm over}\;\;(0,\infty).
\end{eqnarray}
Then, by combining \eqref{lc2} and \eqref{lc3}, we have that for every $\lambda \in (0,1), \;\rho >0$,
\begin{eqnarray}
\|z(\tau_n;\psi,f_{\lambda,n})\|\leqslant \lambda(\rho +r)\quad \mbox{\rm when }\quad  n> n_1(\rho). \label{lc5-1}
\end{eqnarray}
Moreover, by \eqref{1-lc4} and \eqref{lc3}, for every $\lambda \in (0,1)$ we have
\begin{equation*}
\| f_{\lambda,n}\|_{L^\infty((0,\infty);U)}  \leqslant \lambda N^*(\tau)+ (1-\lambda) C \|\psi\|\quad \mbox{\rm for all }\quad n\in \mathbb{N}.
\end{equation*}
This, along with  \eqref{lc1}, indicates that  there exists $\lambda_0 \in (0,1)$ such that
\begin{eqnarray}
\|f_{\lambda_0,n}\|_{L^\infty((0,\infty); U)}< N^*(\tau_n)\quad \mbox{\rm for all }\quad n\in \mathbb{N}. \label{lc6}
\end{eqnarray}
We now fixe $\rho=\rho_0 >0$ such that
$\lambda_0 (r+ \rho_0)\leqslant r$.
Then, by  \eqref{lc5-1}, we have  that
\begin{equation} \label{lc12}
\|z(\tau_n; \psi, f_{\lambda_0,n}) \| \leqslant r, \quad\mbox{\rm when}\quad n> n_1(\rho).
\end{equation}
By taking $n=n_0 >n_1(\rho)$ in \eqref{lc12} and \eqref{lc6}, we are led to  \eqref{marius23},
 where $f_{n_0}=f_{\lambda_0,n_0}$.
Hence,  the map $\tau \mapsto N^*(\tau)$ is left continuous.

To prove the first equality in the second assertion of the proposition, let $(\tau_n)_{n>0}$ be a sequence of positive numbers such that $\displaystyle \lim_{n \to \infty} \tau_n =0$. Suppose by contradiction that there exists a subsequence of $(\tau_n)_{n>0}$, denoted in the same way,  such that
$\displaystyle \lim_{n \to \infty} N^*(\tau_n) \leqslant \rho$ for some $\rho >0$.
For each $n\in \mathbb{N}$, let $f^*_n$ be the extension of the norm optimal control of $(NP)^{\tau_n}$ over $(0,\infty)$ with zero value over $(\tau_n,\infty)$. Then, there exists $\widetilde f\in L^\infty((0,\infty);U)$
and a subsequence of $(f^*_n)_{n\in\mathbb{N}}$, denoted in the same manner, such that
$$
f_n^* \to \widetilde f \quad\mbox{\rm weakly* in }\ \ L^\infty((0,\infty);U)\;\;\mbox{\rm and}\;\;z(\tau_n;\psi,f^*_n) \to z(0;\psi,\widetilde f) \quad\mbox{ \rm weakly in }\ \ X.
$$
Since $z(\tau_n;\psi,f^*_n)\in B(0,r)$ for all $n$, the above two convergence properties imply that $\psi=z(0;\psi,\widetilde f) \in B(0,r)$, which  contradicts the assumption that $\psi \notin B(0,r)$.
 Hence, $\displaystyle \lim_{t \to 0+} N^*(t) = +\infty$.

The second equality in the second assertion of the proposition follows immediately from the definition \eqref{TAU0} of $\widehat{\tau}$ and the continuity of the map $\tau \mapsto N^*(\tau)$.

To prove the third assertion of the proposition,  we first check that $N^*(\tau^*(M))=M$.  If it did not hold, then
we would have either $N^*(\tau^*(M))>M$ for some $M>0$ or $N^*(\tau^*(M))<M$ for some $M>0$. In the case when $N^*(\tau^*(M))>M$ for some $M>0$, any control driving the solution of the system to $B(0,r)$ at time $\tau^*(M)$ would have  its $L^\infty((0,\infty);U)$-norm strictly greater than $M$.
Hence, the problem $(TP)^M$ has no optimal control. However, since $-A:{\cal D}(A)\to X$ generates a contraction semigroup and (\ref{OBSL1}) holds,
we can apply Theorem 3.1 in \cite{PhWZh} to get that  the problem $(TP)^M$ has optimal controls. This leads to a contradiction.
In the second case where $N^*(\tau^*(M))<M$ for some $M>0$, it follows from  the continuity and monotonicity of the map $t \mapsto N^*(t)$ that there exists $\tau_0\in (0, \tau^*(M))$ such that $N^*(\tau_0)=M$. Thus, there is a control $v$ such that
\begin{equation}\label{contradiction2}
\|v\|_{L^\infty((0,\tau_0); L^2(\Omega))}=M\quad \mbox{\rm and}\quad z(\tau_0; \psi,v)\in B(0,r).
\end{equation}
Since $\tau_0<\tau^*(M)$, we are led, from (\ref{contradiction2}), a contradiction to the optimality of $\tau^*(M)$ to the problem $(TP)^M$.
Thus, we have shown that  $N^*(\tau^*(M))=M$ for each $M>0$.
This, in particular, gives  that $N^*(\tau^*(N^*(\tau)))=N^*(\tau)$ for each $\tau \in(0,\widehat{\tau})$. The later, along with  the continuity and monotonicity of $N(\cdot)$,  leads to $\tau^*(N^*(\tau))=\tau$ for all $\tau \in (0,\widehat{\tau})$. Hence, the third assertion of the proposition has been proved. This ends the proof. \qquad \end{proof}

\begin{remark}
We mention that the properties of the minimal time and minimal norm functions derived in the above proposition have been investigated, in the case in which the target is the origin, in Wang and Zuazua \cite{WangZuazua}. Our method differs from the one in \cite{WangZuazua} in the sense that we first study the minimal norm function and then use its properties to obtain the properties of the minimal time function.
\end{remark}

The result below shows that the solution of the  time optimal control to the problem $(TP)^M$ has a simple expression in terms of the minimizer of $J^{\tau^*(M)}$, where $\tau^*(M)$ is the optimal time in $(TP)^M$ and $J^\tau$ has been defined in \eqref{minimization func}.

\

\begin{proposition}\label{equivalence J-T}
With the notation and assumptions of Proposition \ref{NT},
suppose that \eqref{uni-cont-3.5} holds. Let $M>0$ and
let $\widehat \eta$ be the minimizer of $J^{\tau^*(M)}$.
Then the function  $u^*$ defined by
\begin{eqnarray}
u^*(t)= M \frac{B^* \mathbb{T}^*_{\tau^*(M)-t}\widehat \eta}{\|B^* \mathbb{T}^*_{\tau^*(M)-t}\widehat \eta\|_U} \quad \mbox{\rm for a.e. }\quad  t\in(0,\tau^*(M)), \label{bb-op}
\end{eqnarray}
is the time optimal control to the problem  $(TP)^M$.
\end{proposition}
\begin{proof}
Since the uniqueness of the optimal control to the problem $(TP)^M$ has been proved in  Corollary \ref{bb_cor}, we only need to show that
the function $u^*$ defined by (\ref{bb-op}) is an optimal control to the problem $(TP)^M$. To this end, we first claim
\begin{eqnarray}
\widehat \eta \not = 0.\label{marius-1}
\end{eqnarray}
Indeed, if $\widehat \eta=0$, then by the third assertion of Lemma \ref{JT}, we would have that  $\mathbb{T}_{\tau^*(M)}\widehat \psi \in B(0,r)$.
Hence, the null control is the time optimal control to $(TP)^M$. This contradicts to
 the assumption $\psi \not \in B(0,r)$. Consequently, \eqref{marius-1} holds. Hence, $u^*$ is  well defined.

By Proposition \ref{NT}, we have that $N^*(\tau^*(M))=M$. Then by Proposition \ref{JT and NT}, $u^*$ is the norm optimal control of $(NP)^{\tau^*(M)}$, i.e., $z(\tau^*(M);\psi,u^*) \in B(0,r)$ and $ \|u^*\|_{L^\infty((0,\tau^*(M));L^2(\Omega))}=N^*(\tau^*(M))=M$. Consequently, $u^*$ is the time optimal control to the problem $(TP)^M$. This completes the proof. \qquad \end{proof}

\

\begin{remark}\label{time-norm-control} Note that the approximate observability in any time
implies that the problem $(NP)^\tau$ has solutions.
This, as well as Corollary \ref{bb_n_cor}, shows that for each $\tau>0$, the problem $(NP)^\tau$ has a unique optimal control.

To ensure the existence of time optimal controls to the problem $(TP)^M$, we need more assumptions due to the control constraint. More precisely, we have
$$
(\ref{OBSL1})\; \mbox{\rm together with the contractivity of}\; \{\mathbb{T}_t\}_{t\geqslant 0}\;\; \Longrightarrow \;\; (TP)^M\;\mbox{\rm has solutions}.
$$
The later, together with Corollary \ref{bb_cor}, implies that for each $M>0$, $(TP)^M$ has a unique optimal control.

\end{remark}


\section{Proof of Theorem~\ref{th_main}}
\setcounter{equation}{0}
\label{sec_main_proof}

\

In this section we come back to the assumptions and notation in Section \ref{sec_intro}.
In particular, for every $\varepsilon\geqslant 0$ we have that $-A_\varepsilon$ generates a contraction semigroup on $X$, denoted by $\mathbb{T}^\varepsilon$, which is analytic
and exponentially stable (uniformly with respect to $\varepsilon$, see the first estimate in \eqref{unif_decay}).
We recall from Section \ref{sec_intro} that the main object is to study the asymptotic behavior when $\varepsilon\to 0$ of the family of time optimal control problems $\left((TP)^M_\varepsilon\right)_{\varepsilon\geqslant 0}$ defined by (\ref{time-optimal-control1.9}), where $M>0$, $r>0$ and $\psi\not \in B(0,r)$ are given. Moreover, for every $\tau>0$ and $\varepsilon\geqslant 0$ we define the norm optimal control and minimization problems respectively:
\begin{equation}\label{huang5.1}
(NP)^\tau_\varepsilon \quad N_\varepsilon^*(\tau) := \displaystyle\min_{f\in L^{\infty}((0,\tau);U)} \left\lbrace \|f\|_{L^{\infty}((0,\tau);U)}\ |\ \mathbb{T}_t^\varepsilon \psi+\Phi_t^\varepsilon f\in B(0,r)\right\rbrace;
\end{equation}
\vspace{-0.5cm}
\begin{equation}\label{huang5.2}
(JP)^\tau_\varepsilon\qquad\qquad
V^*_{\varepsilon}(T) := \displaystyle\min_{\eta\in X} \left\lbrace J^\tau_\varepsilon(\eta)\right\rbrace,
\end{equation}
\vspace{-0.5cm}
$$
\noindent \mbox{where} \;J^\tau_\varepsilon(\eta) = \frac{1}{2} \left(\int_0^\tau \|B^*\varphi_\varepsilon(t;\eta,\tau) \|_{U}\, {\rm d}t \right)^2 + \left\langle \psi,\varphi_\varepsilon(0;\eta,\tau)\right\rangle+r \| \eta\| \qquad(\eta\in X),
$$
with the notation $ \varphi_\varepsilon(t;\psi,\tau) :=\mathbb{T}_{\tau-t}^{\varepsilon} \psi\ \ (t\in [0, \tau])$.

\begin{remark}\label{added_rem}
It should be pointed out that when assumptions (C1)-(C4) hold, each pair $(A_\varepsilon, B)$, with $\varepsilon\geqslant 0$, satisfies the conditions \eqref{uni-cont-3.5} and \eqref{OBSL1}. Indeed, \eqref{OBSL1} is exactly (C4), whereas \eqref{uni-cont-3.5} follows
from (C4) and the analyticity of the semigroup $\mathbb{T}^\varepsilon$. Therefore, all results obtained in the previous two sections  hold for each pair $(A_\varepsilon, B)$, with $\varepsilon\geqslant 0$.
\end{remark}

Our strategy to show the main results is to build up a strong  convergence of minimizers  of
$(JP)_\varepsilon^{\tau_\varepsilon}$ in $X$ and then to apply the relation of the minimal time and norm optimal
control problems established Section \ref{sec_relation}. More precisely,
an important ingredient of the proof of our main theorem states as follows.

\

\begin{proposition}\label{strong convergence minimizer}
With the above notation, suppose that (C1)-(C4) hold.
Let $(\tau_\varepsilon)_{\varepsilon\geqslant 0}$ be a family of positive numbers such that $\tau_\varepsilon \to \tau_0$, with $\tau_0>0$,
 and let $\widehat\eta_\varepsilon$ be the minimizer of
 $J_{\varepsilon}^{\tau_\varepsilon}$ for each $\varepsilon \geqslant 0$. Then we have $ \widehat \eta_\varepsilon \to \widehat\eta_0$  strongly in $X$.
\end{proposition}

\

To prove the above proposition we use  an essential way in the following lemma, which has many common points with a result in \cite[Ch.4.2]{Zuazua}.

\

\begin{lemma}\label{coercivity}
With the same notation and assumptions as in Proposition \ref{strong convergence minimizer}, we have:\\
\noindent 1. $\displaystyle\liminf_{\varepsilon \to 0+, \|\eta\|\to \infty} \ J_\varepsilon^{\tau_\varepsilon}({\eta}) / \|\eta\| \geqslant r$.

\

\noindent 2. The family $(\widehat \eta_\varepsilon)_{\varepsilon>0}$  is bounded in $X$.
\end{lemma}

\

\begin{proof}
Let $(\varepsilon_j)_{j\in \mathbb{N}}$ be a sequence of numbers such that $\varepsilon_j \to 0$ and let  $(\eta_j)_{j\in \mathbb{N}} $ be a sequence in $X$  such that $\|\eta_j \| \to \infty$. For each $j\in \mathbb{N}$, we set $ I_j= {J_{\varepsilon_j}^{\tau_{\varepsilon_j}}(\eta_j)}/{\|\eta_j\|}$. It suffices to show that  $I_j \geqslant r$ for all $j\in \mathbb{N}$.

Indeed, for each $j\in \mathbb{N}$, we set $\gamma_j =\eta_j/\|\eta_j\|$. It is clear that
\begin{eqnarray}\label{ggg}
I_j= r +\left\langle \psi, \varphi_{\varepsilon_j}(0; \gamma_j,\tau_{\varepsilon_j}) \right\rangle+ \frac{\|\eta_j\|}{2} \left(\int_0^{\tau_{\varepsilon_j}} \|B^*\varphi_{\varepsilon_j}(t;\gamma_j,\tau_{\varepsilon_j})\|_U\, {\rm d}t \right)^2.
\end{eqnarray}
We first consider the case when $ \displaystyle \liminf_{j \to \infty} \left(\int_0^{\tau_{\varepsilon_j}}
\|B^*\varphi_{\varepsilon_j}(t; \gamma_j,\tau_{\varepsilon_j})\|_U\,{\rm d}t \right)^2 >0$.
Since, using the first estimate in \eqref{unif_decay}, we have, for each $j\in \mathbb{N}$, that $\| \varphi_{\varepsilon_j}(0; \gamma_j,\tau_{\varepsilon_j})\|  \leqslant 1$, it follows that  $\displaystyle\liminf_{j \to \infty} I_j =\infty$, which obviously implies $I_j \geqslant r$.

We next consider the case when $ \displaystyle \liminf_{j \to \infty} \left(\int_0^{\tau_{\varepsilon_j}} \|B^*\varphi_{\varepsilon_j}(t;\gamma_j,\tau_{\varepsilon_j})\|_U\, {\rm d}t \right)^2 =0$.
Suppose by contradiction that, up to the extraction of subsequences, we have that
\begin{eqnarray}\label{ggg2}
\int_0^{\tau_{\varepsilon_j}} \|B^*\varphi_{\varepsilon_j}(t; \gamma_j,\tau_{\varepsilon_j})\|_{U}\, {\rm d}t  \to 0, \quad\quad \displaystyle \liminf_{j \to \infty} I_j <r
\end{eqnarray}
and that $\gamma_j \to \widehat \gamma_0$ weakly in $X$, for some $\widehat\gamma_0 \in X$. The last convergence clearly implies, by using the second assertion in Proposition \ref{cons_abs_1}, that for any $\delta \in (0,\tau_0)$,
\begin{eqnarray}
\varphi_{\varepsilon_j}(\cdot;\gamma_j,\tau_{\varepsilon_j}) \to \varphi_0(\cdot;\widehat \gamma_0,\tau_{\varepsilon_j})\quad \mbox{\rm in }\quad C([0,\tau_0-\delta];X). \label{ggg5}
\end{eqnarray}
Combining \eqref{ggg5} and the first result in \eqref{ggg2}, we obtain
$$
\int_0^{\tau_0-\delta} \|B^* \varphi_0(t;\widehat \gamma_0,\tau_{\varepsilon_j})\|_{U}\, {\rm d}t \leqslant \displaystyle \lim_{j \to \infty} \int_0^{\tau_{\varepsilon_j}} \|B^*\varphi_{\varepsilon_j}(t;\gamma_j,\tau_{\varepsilon_j})\|_{U}\, {\rm d}t   =0.
$$
Using assumption (C4) and the analyticity of $\mathbb{T}^\varepsilon$, this implies that $\widehat \gamma_0=0$. Consequently, we have $\gamma_j \to 0$  weakly in $X$. This, together with  \eqref{ggg5}, indicates that
\begin{eqnarray}
 \varphi_{\varepsilon_j}(0;\gamma_j,\tau_{\varepsilon_j}) \to 0\quad \mbox{\rm weakly in }\quad X. \label{ggg7}
\end{eqnarray}
Finally, it follows by \eqref{ggg} and \eqref{ggg7} that $\displaystyle\liminf_{j \to \infty} I_j \geqslant r$.
This contradicts to the second result in \eqref{ggg2}. Hence $I_j \geqslant r$. This ends the proof of the first assertion.

To prove the second assertion of the lemma, we assume, by contradiction, that there exists a sequence $(\varepsilon_j)_{j\in \mathbb{N}}$ such that $\displaystyle \lim_{j \to \infty} \|\widehat \eta_{\varepsilon_j}\| =\infty$. Then, it is clear that from the first assertion of this lemma, we have
$$
0=\displaystyle \liminf_{j \to \infty} \frac{J_{\varepsilon_j}^{\tau_{\varepsilon_j}}(0)}{\| \widehat  \eta_{\varepsilon_j}\|}\geqslant \displaystyle \liminf_{j \to \infty} \frac{J_{\varepsilon_j}^{\tau_{\varepsilon_j}}(\widehat \eta_{\varepsilon_j}) }{\|\widehat \eta_{\varepsilon_j}\|} \geqslant r,
$$
which leads to a contradiction. Hence, $(\widehat \eta_\varepsilon)_{\varepsilon>0}$ is bounded. This ends the proof. \qquad \end{proof}

\

{\em Proof of Proposition \ref{strong convergence minimizer}.}
Let $(\varepsilon_j)$ be a sequence with $\varepsilon_j \to 0$. We will show  that
\begin{eqnarray}\label{SC1}
\widehat \eta_{\varepsilon_j} \to \widehat\eta_0\quad \mbox{\rm weakly in }\;\; X \qquad \mbox{and} \qquad \displaystyle \lim_{j \to \infty} \|\widehat \eta_{\varepsilon_j} \| = \| \widehat \eta_0\|.
\end{eqnarray}
Combining the above two assertions, the strong convergence stated in Proposition \ref{strong convergence minimizer} follows at once.

We first prove the first convergence in \eqref{SC1}.
Since, the sequence $(\widehat \eta_{\varepsilon_j})_{j\in \mathbb{N}} $ is bounded (Lemma \ref{coercivity}), it suffices to prove that any weakly convergent subsequence
$(\widehat \eta_{\varepsilon_{j_k}})_{k\in\mathbb{N}}$  of $(\widehat \eta_{\varepsilon_j})_{j\in \mathbb{N}} $ has the same weak limit $\widehat\eta_0$.
To prove this, assume that for some $\widetilde\eta\in X$,
\begin{eqnarray}
\displaystyle \lim_{k\to \infty} \widehat \eta_{\varepsilon_{j_k}} = \widetilde\eta\quad \mbox{\rm weakly in }\quad X. \label{SC2}
\end{eqnarray}
Since $J_0^\tau$ has a unique minimizer (see Lemma \ref{JT}), it suffices to show that
\begin{eqnarray}
J_0^{\tau_0}(\widetilde \eta)\leqslant J_0^{\tau_0}(\widehat \eta_0). \label{SC4}
\end{eqnarray}
In fact, by the second assertion of Proposition \ref{cons_abs_1} and \eqref{SC2}, it follows that
\begin{equation}\label{SC6}
\noindent \langle \psi, \varphi_0(0;\widetilde \eta,\tau_0)\rangle = \displaystyle \lim_{k\rightarrow\infty} \langle \psi,
\varphi_{\varepsilon_{j_k}}(0;\widehat \eta_{\varepsilon_{j_k}},\tau_{\varepsilon_{j_k}})\rangle \quad \mbox{and} \quad r\|\widetilde \eta\| \leqslant \displaystyle \liminf_{k \to \infty}   r\|\widehat \eta_{\varepsilon_{j_k}}\|.
\end{equation}
By the first estimate in \eqref{unif_decay}, for every $t\in[0,\tau_0]$, we have $\|\varphi_0(t;\widetilde \eta,\tau_0))\| \leqslant  \|\widetilde \eta \|$. Thus, given  $\sigma>0$, there exists $\delta_0=\delta_0(\sigma)>0$ such that when $\delta \leqslant \delta_0$, we have
$$
 \left(\int_0^{\tau_0} \|B^*\varphi_0(t;\widetilde \eta,\tau_0) \|_{U}\, {\rm d}t  \right)^2  \leqslant \left(\int_0^{{\tau_0}-\delta} \|B^* \varphi_0(t;\widetilde \eta,\tau_0) \|_{U}\, {\rm d}t  \right)^2 +\sigma.
$$
together with $\tau_{\varepsilon_{j_k}} > \tau_0- \delta$.
According to the second assertion of Proposition \ref{cons_abs_1}, we know that for every $\delta \in (0,{\tau_0})$,
$$
\varphi_{\varepsilon_{j_k}}(\cdot;\widehat \eta_{\varepsilon_{j_k}},\tau_{
\varepsilon_{j_k}}) \to \varphi_0(\cdot;\widetilde \eta,\tau_0)\quad \mbox{\rm in }\quad C([0,\tau_0-\delta];X).
$$
Hence, we have
$$
\left(\int_0^{\tau_0} \|B^*\varphi_0(t;\widetilde \eta,\tau_0) \|_{U}\, {\rm d}t  \right)^2  \leqslant\displaystyle \liminf_{k\to \infty} \left(\int_0^{\tau_{\varepsilon_{j_k}}} \|B^*\varphi_{\varepsilon_{j_k}}(t;\widehat \eta_{\varepsilon_{j_k}},\tau_{\varepsilon_{j_k}}) \|_{U}\, {\rm d}t  \right)^2 +\sigma.
$$
Since $\sigma$ can be chosen arbitrarily, it follows that
\begin{equation}
\left(\int_0^{\tau_0} \|B^*\varphi_0(t;\widetilde \eta,\tau_0) \|_{U}\, {\rm d}t  \right)^2 \leqslant\displaystyle \liminf_{k\to \infty} \left(\int_0^{\tau_{\varepsilon_{j_k}}} \|B^*\varphi_{\varepsilon_{j_k}}(t;\widehat \eta_{\varepsilon_{j_k}},\tau_{\varepsilon_{j_k}})\|_{U}\, {\rm d}t  \right)^2 \label{SC8}.
\end{equation}
From the definitions of $J_0^{\tau_0}(\widetilde \eta)$ and of $J_{\varepsilon_{j_k}}^{\tau_{\varepsilon_{j_k}}}(\widetilde \eta_{\varepsilon_{j_k}})$, as well as from \eqref{SC6}, \eqref{SC8} and the optimality of $\widehat \eta_{\varepsilon_{j_k}}$ to $(JP)_{\varepsilon_{j_k}}^{\tau_{\varepsilon_{j_k}}}$, we obtain that
\begin{eqnarray}
J_0^{\tau_0}(\widetilde \eta) \leqslant \displaystyle \liminf_{k \to \infty} J^{\tau_{\varepsilon_{j_k}}}_{\varepsilon_{j_k}}(\widehat \eta_{\varepsilon_{j_k}}) \leqslant \displaystyle \liminf_{k \to \infty} J^{\tau_{\varepsilon_{j_k}}}_{\varepsilon_{j_k}}(\widehat \eta_0). \label{SC9}
\end{eqnarray}

We now prove that
\begin{eqnarray}
\displaystyle \liminf_{k\to \infty} J^{\tau_{\varepsilon_{j_k}}}_{\varepsilon_{j_k}}(\widehat \eta_0) = J_0^{\tau_0}({\widehat \eta_0}). \label{SC16}
\end{eqnarray}
Firstly, we remark that, according to the second assertion of Proposition \ref{cons_abs_1}, we have
\begin{eqnarray}
\left\langle \psi,\varphi_{\varepsilon_{j_k}}(0;\widehat \eta_0,\tau_{\varepsilon_{j_k}}) \right\rangle \to \left\langle \psi, \varphi_0(0;\widehat \eta_0,\tau_0)\right\rangle \label{SC11}.
\end{eqnarray}
We next claim that
\begin{eqnarray}
\int_0^{\tau_{\varepsilon_{j_k}}} \|B^*\varphi_{\varepsilon_{j_k}}(t;\widehat \eta_0,\tau_{\varepsilon_{j_k}})\|_{U}\, {\rm d}t  \to \int_0^{\tau_0} \|B^*\varphi_0(t;\widehat \eta_0,\tau_0) \|_{U}\, {\rm d}t . \label{SC12}
\end{eqnarray}
Indeed, given $\delta \in (0,{\tau_0})$, by using the first estimate in \eqref{unif_decay}, we obtain that there is a natural number $k$  such that $\tau_{\varepsilon_{j_k}}>{\tau_0}-\delta$ and such that
\begin{align*}
C_k&:= \left| \int_0^{\tau_{\varepsilon_{j_k}}} \|B^*\varphi_{\varepsilon_{j_k}}(t;\widehat \eta_0,\tau_{\varepsilon_{j_k}}) \|_{U}\, {\rm d}t  - \int_0^{\tau_0} \|B^*\varphi_0(t;\widehat \eta_0,\tau_0) \|_{U}\, {\rm d}t  \right|
\\
&\leqslant \int_0^{{\tau_0}-\delta} \| B^*(\varphi_{\varepsilon_{j_k}}(t;\widehat \eta_0,\tau_{\varepsilon_{j_k}})-\varphi_0(t;\widehat \eta_0,\tau_0))\|_U\, {\rm d}t
+2\int_{{\tau_0}-\delta}^{\max(\tau_0,\tau_{\varepsilon_{j_k}})}  \| B^*\widehat\eta_0\|_U\,{\rm d}t.
\end{align*}
This yields  that
\begin{equation}\label{SC14}
C_k \leqslant \int_0^{{\tau_0}-\delta} \|B^*( \varphi_{\varepsilon_{j_k}}(t;\widehat \eta_0,\tau_{\varepsilon_{j_k}})-\varphi_0(t;\widehat \eta_0,\tau_0))\|_U\, {\rm d}t + \delta_k,
\end{equation}
where $\delta_k\rightarrow 0$ as $k\rightarrow\infty$.
From \eqref{SC14} and the second assertion in Proposition \ref{cons_abs_1}, \eqref{SC12} follows at once.
Moreover, using the definitions of $J^{\tau_{\varepsilon_{j_k}}}_{\varepsilon_{j_k}}(\widehat \eta_0) $ and of $J^{\tau_0}_0(\widehat \eta_0) $, as well as \eqref{SC11}-\eqref{SC12}, we obtain \eqref{SC16}. Inequality
\eqref{SC4}  is then deduced from \eqref{SC9} and \eqref{SC16}. As already mentioned, \eqref{SC4} implies the first convergence in \eqref{SC1}.

We are now on the position to show the second convergence in \eqref{SC1}.
By the second assertion of Proposition \ref{cons_abs_1} and the first convergence of \eqref{SC1}, we have:
\begin{equation}\label{SC18}
\varphi_{\varepsilon_j}(0;\widehat \eta_{\varepsilon_j},\tau_{\varepsilon_j}) \to \varphi_0(0;\widehat \eta_0,\tau_0) \quad \mbox{and} \quad \displaystyle \liminf_{j \to \infty} \| \widehat\eta_{\varepsilon_j}\| \geqslant \|\widehat \eta_0\|.
\end{equation}

Two observations are given in order. Firstly,
 there exists $C>0$ such that
\begin{align}
\left\lbrace \begin{array}{cc}
\|\varphi_0(t;\widehat \eta_0,\tau_0)\|_{C((0,\tau_0);X)}\leqslant \|\widehat \eta_0 \|\leqslant C & \qquad( t\in[0,\tau]),
\\
\|\varphi_{\varepsilon_j}(t;\widehat \eta_{\varepsilon_j},\tau_{\varepsilon_j})\|_{C((0,\tau_0);X)}\leqslant \|\widehat \eta_{\varepsilon_j} \|\leqslant C & \qquad( t\in [0,\tau_{\varepsilon_j}]).
\end{array} \right. \label{SC21}
\end{align}
Secondly,  by the first convergence in \eqref{SC1} and the second assertion of Proposition \ref{cons_abs_1}, for every $\delta\in(0,\tau_0)$, we have
\begin{eqnarray}
\varphi_{\varepsilon_j}(\cdot;\widehat \eta_{\varepsilon_j},\tau_{\varepsilon_j})\to \varphi_{\varepsilon_j}(\cdot;\widehat \eta_0,\tau_0)\quad \mbox{\rm in }\quad C([0,\tau_0-\delta);X). \label{SC22}
\end{eqnarray}
By (\ref{SC21}) and (\ref{SC22}), we can use the same way to show (\ref{SC12}) to prove that
\begin{eqnarray}
\int_0^{\tau_{\varepsilon_j}} \|B^*\varphi_{\varepsilon_j}(t;\widehat \eta_{\varepsilon_j},\tau_{\varepsilon_j})\|_{U}\, {\rm d}t  \to \int_0^{\tau_0} \|B^* \varphi_0(t;\widehat \eta_0,\tau_0)\|_U\, {\rm d}t  \label{SC20}.
\end{eqnarray}
From the definitions of $J^{\tau_{\varepsilon_j}}_{\varepsilon_j}(\widehat\eta_{\varepsilon_j})$, $J^{\tau_0}_0(\widehat\eta_0)$, as well as \eqref{SC18}-\eqref{SC20}, we have
$$
J_0^{\tau_0}(\widehat \eta_0) \leqslant\displaystyle \liminf_{j \to \infty} J^{\tau_{\varepsilon_j}}_{\varepsilon_j}(\widehat \eta_{\varepsilon_j}),
$$
which, combined with the optimality of $\widehat \eta_{\varepsilon_j}$  leads to
$$
\displaystyle\liminf_{j \to \infty} J^{\tau_{\varepsilon_j}}_{\varepsilon_j}(\widehat\eta_{\varepsilon_j}) \leqslant \displaystyle\liminf_{j \to \infty} J^{\tau_{\varepsilon_j}}_{\varepsilon_j}(\widehat \eta_0).
$$
From the above two inequalities and \eqref{SC16}, we have
$$
J^{\tau_0}_0(\widehat \eta_0)\leqslant \displaystyle \liminf_{j \to \infty} J^{\tau_{\varepsilon_j}}_{\varepsilon_j}(\widehat\eta_{\varepsilon_j}) \leqslant J^{\tau_0}_0(\widehat \eta_0),
$$
which clearly implies that
\begin{equation}\label{SC26}
J_0^{{\tau_0}}(\widehat \eta_0)= \displaystyle \liminf_{j \to \infty} J^{\tau_{\varepsilon_j}}_{\varepsilon_j}(\widehat\eta_{\varepsilon_j}).
\end{equation}
Finally,  from \eqref{SC26}, \eqref{SC18} and \eqref{SC20}, the second convergence in \eqref{SC1} follows at once.
This ends the proof Proposition \ref{strong convergence minimizer}. $\qquad \square$

We next deduce from Proposition \ref{strong convergence minimizer} the following corollary which will be used in the proof of the convergence of the time optimal control.

\

\begin{corollary}\label{lemme N convergence}
With the same notation and assumptions as those in Proposition \ref{strong convergence minimizer}, we have that
 \begin{eqnarray}\label{V convergence}
\displaystyle \lim_{\varepsilon \to 0+}N^*_{\varepsilon}(\tau)= N_0^*(\tau)\;\;\mbox{\rm and}
\;\; \displaystyle \lim_{\varepsilon \to 0+}V^*_{\varepsilon}(\tau)= V_0^*(\tau) \ \ \ (\tau\in (0, \widehat{\tau})).
\end{eqnarray}
(Here, $N^*_{\varepsilon}(\tau)$ and $V^*_{\varepsilon}(\tau)$ are defined by (\ref{huang5.1}) and (\ref{huang5.2}) respectively.)
\end{corollary}

\begin{proof}
By  \eqref{SC26}, we have that
\begin{eqnarray}
J_0^{\tau}(\widehat \eta_0)=  \displaystyle \liminf_{\varepsilon \to 0+} J^{\tau}_{\varepsilon}(\widehat \eta_{\varepsilon}),
\end{eqnarray}
where $\widehat\eta_{\varepsilon}$, with $\varepsilon\geqslant 0$, is the minimizer of  $J_\varepsilon^\tau$.
Meanwhile, we obtain from (\ref{huang5.1}) and (\ref{huang5.2}) that
$ V^*_{\varepsilon}(\tau)=  J^{\tau}_{\varepsilon}(\widehat \eta_\varepsilon)$  and $ V_0^*(\tau)=J_0^{\tau}(\widehat \eta_0)$. These, along with the last assertion in Proposition \ref{JT and NT}, yield to \eqref{V convergence} and ends the proof. \qquad \end{proof}

\

We are now in a position to prove our main result, i.e., Theorem~\ref{th_main}.

\

{\em Proof of Theorem~\ref{th_main}.}
To prove \eqref{time_conv}, we observe from the first estimate in \eqref{unif_decay} that there exists $\hat{\tau}>0$ such that $\displaystyle \mathbb{T}_{\hat\tau}^{\varepsilon}\psi\in B(0,r)\;\;\mbox{\rm for all}\;\;\varepsilon\geqslant 0$.
 Hence,  $\tau^*_\varepsilon(M) \leqslant \hat{\tau}$ for all $\varepsilon\geqslant 0$.
Therefore, all cluster points of $(\tau^*_\varepsilon(M))_{\varepsilon>0}$ are finite. Let $\hat{t}$ be a cluster point of $(\tau^*_\varepsilon(M))_{\varepsilon>0}$, i.e., there is  a sequence $(\varepsilon_n)_{n\in \mathbb{N}}$ of positive numbers with
$\varepsilon_n \to 0$ such that  $ \hat{t}=\displaystyle \lim_{n \to \infty} \tau^*_{\varepsilon_n}(M)$.
For each $n$, we let $\tilde{u}_{\varepsilon_n}$ be the extension of the optimal control $u^*_{\varepsilon_n}$ over $(0,\infty)$ with zero value over $(\tau^*_{\varepsilon_n}(M),\infty)$.
Then, up to the extraction of a subsequence, we have
\begin{eqnarray}
\tilde{u}_{\varepsilon_n} \to \tilde{u}\quad \mbox{\rm  weakly* in }\quad L^\infty((0,\hat{t});U). \label{tt1}
\end{eqnarray}
This, combined with  Corollary \ref{clasica_1},  leads to
\begin{eqnarray}
\mathbb{T}_{\hat t}^{\varepsilon_n} \psi+ \Phi_{\hat t}^{\varepsilon_n} \tilde{u}_{\varepsilon_n} \to\mathbb{T}_{\hat t}\psi+ \Phi_{\hat t}\tilde{u}\quad \mbox{\rm weakly in }\quad X. \label{tt2}
\end{eqnarray}
From \eqref{tt1} and \eqref{tt2}, we see that $ \| \tilde{u} \|_{L^\infty((0,\hat{t});U)} \leqslant M$ and that $ \|\mathbb{T}_{\hat t} \psi+ \Phi_{\hat t} \tilde{u}\| \leqslant r$.
Consequently, $\hat{t}\geqslant \tau_0^*(M)$ and $\tilde{u}$ is an admissible control for $(TP)_0^M$.

At this point we assume, by contradiction, that $\hat{t} > \tau_0^*(M)$. Then there exists $\delta>0$ and $N(\delta)\in \mathbb{N}$
such that
$$
\tau^*_{\varepsilon_n}(M) > \tau_0^*(M)+\delta \qquad \qquad(n\geqslant N(\delta)).
$$
The above estimate and Proposition \ref{NT} imply that
\begin{equation}\label{tt3}
M=N^*_{\varepsilon_n}(\tau^*_{\varepsilon_n}(M)) <N^*_{\varepsilon_n}(\tau_0^*(M)+\delta) < N_{\varepsilon_n}^*(\tau_0^*(M)).
\end{equation}
On the other hand, from Corollary \ref{lemme N convergence} and Proposition~\ref{NT}, it follows that
\begin{equation}\label{tt4}
N^*_{\varepsilon_n}(\tau_0^*(M)+\delta) \to N_0^*(\tau_0^*(M)+\delta) \quad \mbox{and} \quad N_{\varepsilon_n}^*(\tau_0^*(M)) \to N_0^*(\tau_0^*(M))=M.
\end{equation}
Combining \eqref{tt3} and \eqref{tt4} leads to $ M\leqslant N_0^*(\tau_0^*(M)+\delta) \leqslant N_0^*(\tau_0^*(M))=M$,
which implies that $ N_0^*(\tau_0^*(M)+\delta)=N_0^*(\tau_0^*(M))$. This contradicts to the strict monotonicity of $N_0^*(\cdot)$. Hence $\hat{t}=\tau_0^*(M)$, which clearly implies \eqref{time_conv}.

\

To prove \eqref{con_conv}, let $(\varepsilon_n)_{n >0}$ be a family such that $\varepsilon_n \to 0$. For each $n$, denote $\tilde{u}_{\varepsilon_n}$ the extension of the optimal control $u^*_{\varepsilon_n}$ over $(0,\infty)$, with zero value over $(\tau^*_{\varepsilon_n}(M),\infty)$. It suffices to show
\begin{eqnarray}
\tilde{u}_{\varepsilon_n} \to u_0^*\quad \mbox{\rm stongly in }\quad L^2((0,\tau_0^*(M));U). \label{ci1}
\end{eqnarray}
It is clear that up to the extraction of a subsequence, there exists a subsequence $(\varepsilon_{n_k})_{k\in\mathbb{N}}$ such that
$$
\tilde{u}_{\varepsilon_{n_k}}\to \widetilde u_0\quad \mbox{\rm  weakly* in }\quad L^\infty((0,\tau^*_0(M));U)
$$
for some $\widetilde u_0 \in L^\infty((0,\tau^*_0(M));U)$. This, combined with \eqref{time_conv} and with Corollary \ref{clasica_1}, yields that
$$
\|\mathbb{T}_{\tau_0^*(M)}^{\varepsilon_n} \psi+ \Phi_{\tau_0^*(M)}^{\varepsilon_n} \widetilde u_0 \| \leqslant r,\quad \mbox{and} \quad \|\widetilde u_0\|_{L^\infty((0,\tau^*_0(M));U)}  \leqslant M.
$$
The above two estimates imply that $\widetilde u_0=u_0^*$. In particular, we have
\begin{eqnarray}
\widetilde u_{\varepsilon_{n_k}} \to u_0^*\quad \mbox{\rm  weakly in }\quad L^2((0,\tau_0^{*}(M));U). \label{ci2}
\end{eqnarray}
Meanwhile, it follows  by Lemma \ref{equivalence J-T} that
$$
\|\widetilde u_{\varepsilon_{n_k}}\|=M\;\; \mbox{\rm for a.e. }\;\; t\in (0,\tau^*_{\varepsilon_{n_k}}(M))
\;\;\mbox{\rm and}\;\;\|u_0^*\|=M\;\; \mbox{\rm for a.e. }\;\; t\in (0,\tau_0^*(M)).
$$
Using again \eqref{time_conv} we obtain that
\begin{eqnarray*}
\| \widetilde u_{\varepsilon_{n_k}}\|_{ L^2((0,\tau_0^*(M));U)} \to \| u_0^*\|_{ L^2((0,\tau_0^*(M));U)}. \label{ci3}
\end{eqnarray*}
The last estimate, combined with   \eqref{ci2}, implies \eqref{ci1}.

\

To prove \eqref{con_conv_unif},
let $\widehat \eta_\varepsilon$ be the minimizer of
$J_\varepsilon^{\tau_\varepsilon^*(M)}$ for each $\varepsilon \geqslant 0$.
We first claim that
$\widehat{\eta_\varepsilon} \not =0$ for each $\varepsilon \geqslant 0$. Suppose by contradiction that
$\widehat{\eta_\varepsilon} =0$ for some $\varepsilon\geqslant 0$. Then by the third assertion of  Lemma \ref{JT}, we have that
$\mathbb{T}_{\tau^*_\varepsilon}\psi\in B(0,r)$. Hence, the null control is optimal to the problem $(TP)^M_\varepsilon$, which leads to a
contradiction, since $\psi\notin B(0,r)$. For the sake of simplicity, we denote, for each $\varepsilon\geqslant 0$,
$$
\widehat \varphi_\varepsilon(t)=\varphi_\varepsilon(t;\widehat \eta_\varepsilon,
\tau^*_\varepsilon(M)) \qquad \qquad(t\in [0, \tau^*_\varepsilon(M)]).
$$
Using \eqref{time_conv} and the second assertion in Proposition \ref{cons_abs_1}, we have
that
\begin{eqnarray} \label{cii1}
\widehat \varphi_\varepsilon \to \widehat \varphi_0 \quad \mbox{\rm strongly  in }\quad C([0,\tau_0^*(M)-\delta];X),
\end{eqnarray}
for every $\delta\in(0,\tau_0^*(M))$.
From Proposition \ref{equivalence J-T}, we have that
$$
u^*_\varepsilon(t)= M \frac{ B^* \widehat \varphi_\varepsilon(t) }{ \|B^* \widehat \varphi_\varepsilon(t)\|_U}
\qquad\qquad(t\in (0,\tau^*_\varepsilon(M))),
$$
\vspace{-0.5cm}
$$
u_0^*(t)= M \frac{ B^* \widehat\varphi_0(t)}{ \|B^* \widehat \varphi_0(t)\|_U}  \qquad\qquad(t\in (0,\tau_0^*(M))).
$$
Given $\delta \in (0,\tau_0^*(M))$, the above equalities imply  that for each $t\in[0,\tau_0^*(M)-\delta]$,
\begin{equation}
u^*_\varepsilon(t)-u_0^*(t)= G_\varepsilon(t)\left(
\| B^*\widehat\varphi_0(t)\|_U\, B^* \widehat \varphi_\varepsilon(t)
 - \| B^*\widehat\varphi_\varepsilon(t)\|_U\, B^*  \widehat\varphi_0(t)\right). \label{cii2}
\end{equation}
where $G_\varepsilon(t)=\displaystyle\frac{M}{\| B^* \widehat\varphi_\varepsilon(t)\|_U \| B^* \widehat\varphi_0(t)\|_U}$.
We next claim that there exists $C_\delta >0$ with
\begin{eqnarray}
\| B^* \widehat\varphi_0(t)\|_U \geqslant C_\delta \qquad\qquad(t\in[0,\tau_0^*(M)-\delta]).\label{cii3}
\end{eqnarray}
By contradiction, suppose that \eqref{cii3} failed. Then, there would be a sequence
$(t_n)_{n>0} \subset [0,\tau_0^*(M)-\delta]$ such that
\begin{equation}\label{5.50}
\|  B^* \widehat\varphi_0(t_n)\|_U < \frac{1}{n} \qquad(n\in \mathbb{N}).
\end{equation}
Since $t_n\in [0,\tau_0^*(M)-\delta]$, there is a subsequence, denoted in the same way, such that
$t_n \to s_0$ for some $s_0\in[0,\tau_0^*(M)-\delta]$.
This, along with (\ref{5.50}), yields that  $\|B^* \widehat\varphi_0(s_0)\|_U=0$. Then,
by the assumption \eqref{Lin_unique}, we have that $ \widehat \eta_0=0$, which, combined with the third assertion of Lemma \ref{JT}, leads to a contradiction. Hence, \eqref{cii3} holds.

At this point we note that \eqref{cii1} implies that
$$
\| B^* \widehat\varphi_\varepsilon(t)\|_U \to \|  B^*\widehat\varphi_0(t)\|_U\quad
\mbox{\rm uniformly for }\quad t\in[0,\tau_0^*(M)-\delta].
$$
This, along with \eqref{cii3}, yields that there exits $\varepsilon_0>0 $ such that when
$\varepsilon \leqslant \varepsilon_0$,
\begin{eqnarray}
\| B^* \widehat\varphi_\varepsilon(t)\|_U \|B^* \widehat\varphi_0(t)\|_U \geqslant \frac{C_\delta^2}{2}
 \qquad(t\in [0,\tau_0^*(M)-\delta]). \label{cii4}
\end{eqnarray}
Finally, using \eqref{cii4}, \eqref{cii2} and \eqref{cii1} it follows that
$$
u^*_\varepsilon \to u_0^*\quad \mbox{\rm strongly in }\quad L^\infty((0,\tau_0^*(M)-\delta);U),
$$
which ends the proof. \qquad $\square$

\section{Examples, comments and open questions}
\label{sec_final}
\setcounter{equation}{0}

\

In the first part of this section we give two examples of application of Theorem \ref{th_main} to
families (depending on a small parameter $\varepsilon$) of systems governed by parabolic partial
differential equations. The main example concerns the homogenization of parabolic equations, as already
stated in the Introduction. However, for the sake of simplicity, we begin by a more elementary example, borrowed from \cite{Yu}.

\

\begin{example}
Let  $\Omega \subset \mathbb{R}^N$ be a bounded open set with a $\partial \Omega$ of class $C^2$. For each
$\varepsilon\geqslant 0$, $\psi\in L^2(\Omega)$ and $u\in L^\infty(0,\infty);L^2(\Omega))$ we denote by
$z_\varepsilon:=z_\varepsilon (t;\psi,u)$ the solution of the initial and boundary value problem
\begin{equation}\label{heat_osc_simple}
\dot z_\varepsilon(x,t)-\,  \triangle z_\varepsilon(x,t)-a_\varepsilon(x)z(x,t)=\chi_\omega(x) u(x,t) \quad\qquad(x\in \Omega,\ \ t\geqslant 0),
\end{equation}
\vspace{-0.75cm}
\begin{equation}\label{bound_osc_simple}
z_\varepsilon(x,t)=0 \qquad(x\in\partial\Omega,\ \ t\geqslant 0),
\end{equation}
\vspace{-0.75cm}
\begin{equation}\label{init_osc_simple}
z_\varepsilon(x,0)=\psi(x) \qquad(x\in \Omega),
\end{equation}
where  $a_\varepsilon \in L^\infty(\Omega)$ and $\omega\subset \Omega$ is an open and nonempty subset with its
characteristic function $\chi_\omega$. We further assume that
$$
\displaystyle \lim_{\varepsilon \to 0+}\|a_\varepsilon -a_0\|_{L^\infty(\Omega)}=0 \quad \mbox{and} \quad \|a_\varepsilon\|_{L^\infty(\Omega)} \leqslant \lambda_1\;\; (\varepsilon\geqslant 0),
$$
where $\lambda_1$ is the first eigenvalue of $-\triangle$ with the homogeneous Dirichlet boundary condition.
Given $M,\ r>0$, for each $\varepsilon\geqslant 0$, we consider the time optimal control problem:

\

\noindent ${(TPH1)^M_\varepsilon}\;\;\;\;
\tau^*_{\varepsilon}(M):= \displaystyle\min \left\lbrace t \ |\ \|z_\varepsilon(t;\psi,u))\|_{L^2(\Omega)}\leqslant r\right\rbrace $,
 where the minimum  is taken over the closed ball in $L^\infty((0,\infty); L^2(\Omega))$,
 centered at the origin with radius $M$.

\

\noindent With the above notation and assumptions, it is clear that for every $\varepsilon\geqslant 0$
 the problem $(TPH1)^M_\varepsilon$ admits a unique solution
$(\tau_\varepsilon^*(M),u_\varepsilon^*(M))$. We have the following result, which has been proved,
with a different method, in \cite{Yu}.

\

\begin{proposition}\label{corollary6.2}
With the notation and assumptions above, let, for every $\varepsilon\geqslant 0$, $(\tau_\varepsilon^*(M),u^*_{\varepsilon}(M))$
 be the the solution of the time optimal control problem ${(TPH1)^M_\varepsilon}$. Then we have that
$\displaystyle \tau^*_\varepsilon(M) \to \tau_0^*(M)$, and  $u^*_{\varepsilon} \to u_0^*$ strongly in $L^2((0,\tau_0^*(M));L^2(\Omega))$.

Moreover, for every $\delta\in (0,\tau^*_0(M))$, we have $u^*_{\varepsilon} \to u_0^*$ strongly in $ L^\infty((0,\tau_0^*(M)-\delta);L^2(\Omega))$.
\end{proposition}

\begin{proof}
The results come directly from Theorem \ref{th_main}, with the following choice of spaces and operators:
$$
X=L^2(\Omega),\quad  X_1 = H^2(\Omega)\cap H_0^1(\Omega), \quad
 A_\varepsilon = - \triangle -a_\varepsilon\ \ (\varepsilon\geqslant 0), \ \ U=L^2(\Omega) ,\quad B=\chi_\omega.
$$
Indeed, it can be easily checked that assumptions (C1)-(C4), as well as \eqref{Lin_unique} are all satisfied for the above choice of $X,\ U,\ (A_\varepsilon)$ and $B$. \qquad \end{proof}
\end{example}

\

Our main example of application of Theorem \ref{th_main}, already mentioned in Section \ref{sec_intro}, is described below.

\

\begin{example}
As mentioned in the Introduction, our main result in Theorem \ref{th_main} can be applied to the family of time optimal control problems  ${(TPH)^M_\varepsilon}$ (introduced in Section \ref{sec_intro}), with $\varepsilon\geqslant 0$, provided that we make the appropriate choice
of spaces and operators. More precisely, we set
$$
X=L^2(\Omega), \qquad X_1=H^2(\Omega)\cap H_0^1(\Omega),
$$
For each $\varepsilon>0$, the operator $A_\varepsilon\in {\cal L}(X_1,X)$ is  defined
by
$$
(A_\varepsilon\psi)(x) = -{\rm div}\, \left(a\left(\frac{x}{\varepsilon}\right)\nabla \psi\right)
\qquad(\psi\in X_1, \ \ x\in \Omega),
$$
where $a=(a_{ij})_{1\leqslant i, j\leqslant n}$ is a matrix-valued function satisfying the assumptions
stated at the beginning of Section \ref{sec_intro}.
The operator $A_0$ is defined by
$$
(A_0 \psi)(x) = -{\rm div}\, \left(a_0\nabla \psi\right)
\qquad(\psi\in X_1, \ \ x\in \Omega),
$$
where the matrix $a_0$ is given, according to homogenization theory (see, for instance,
Allaire \cite{All_rev} or Cioranescu and Donato  \cite[Theorem 6.1]{Doina}), by
$$
a_0 \lambda = \int_{[0,1]^n} a(y)\nabla w_\lambda(y)\, {\rm d}y \qquad\qquad(\lambda\in \mathbb{R}^n).
$$
Here $w_\lambda:\mathbb{R}^n\to \mathbb{R}$ is defined, for each $\lambda\in \mathbb{R}^n$, as the solution of
$$
\left\{\begin{array}{lcr} -{\rm div}\, (a(y)\nabla w_\lambda(y))=0 &\ & y\in [0,1]^n,\\
\ &\ &\ \\
w_\lambda (y+e_j) -\lambda\cdot (y+e_j)=w_\lambda (y) -\lambda\cdot y,&\ & (j\in \{1,\dots,n\},\ y\in \mathbb{R}^n),\\
\ &\ &\ \\
\displaystyle\int_{[0,1]^n} (w_\lambda (y) -\lambda\cdot y)\, {\rm d} y=0, &\ &\
\end{array}\right.
$$
where $(e_j)_{1\leqslant j\leqslant n}$ is the standard basis in $\mathbb{R}^n$.
According to standard results (see, for instance \cite[Corollary 6.10, Proposition 6.12]{Doina}),
the matrix $a_0$ is strictly positive.

We also introduce the input space and the control operator by setting
$$
U=L^2(\Omega) \qquad \mbox{\rm and }\quad B= \chi_\omega.
$$

\

{\em Proof  of Proposition \ref{corollary6.4}.}
In order to apply the results from Theorem \ref{th_main} we first verify that assumptions (C1)-(C4) hold with $X,\ U, \ (A_\varepsilon)_{\varepsilon\geqslant 0}$ and $B$ chosen as above. Firstly, using the compactness of the embedding $H_0^1(\Omega)$ into $L^2(\Omega)$ and standard calculations,
it can be easily checked that the above defined family  $(A_\varepsilon)_{\varepsilon\geqslant 0}$
satisfies assumptions (C1) and (C2). The fact that (C3) is also satisfied follows from \cite[Theorem 6.1]{Doina} and again from the compactness of the embedding $H_0^1(\Omega)\subset L^2(\Omega)$. Assumption (C4) is the standard $L^\infty$-null controllability of parabolic equation, which has been proven, for instance, in Wang \cite{wang_linf}.

Moreover, for every $\varepsilon \geqslant 0$, the assumption \eqref{Lin_unique} is the strong unique continuation property, which is known to hold for the considered equations
(see, for instance, Lin \cite{Lin}).
\qquad
$\square$
\end{example}

\

We end up this paper with a series of comments and by stating some open questions.
Firstly, as mentioned above, the strong unique continuation assumption \eqref{Lin_unique} holds for parabolic equations with control distributed in an open set inside the domain but it does not hold, in general, in the case of boundary control. This is the main obstacle in extending our results on $L^\infty$ convergence  to boundary control problems,
for which the control operator is no longer bounded. However, we think that the other convergence properties (for optimal time and optimal controls) can be extended to the case of unbounded control operators.
Note that the perturbation analysis in Section \ref{sec_approx} seems easy to be extended for unbounded control operators (so to boundary control problems).

An interesting problem is the extension of the results in this work to the case
of more general targets (not only closed balls centered in the origin). If we consider,
for instance, targets which are closed balls centered at an arbitrary point of the state space, then our method to
establish the equivalence of time and norm optimal control problems is failing. The case of point targets seems also quite delicate.
This is due, in particular, to  the fact that, for parabolic PDEs, observability estimates which are uniform with respect to $\varepsilon$ are known only in one space dimension, see \cite{Lop_Zuaz}.

\section{Appendix}

\

In this section we provide a simple proof of the maximum principle for time and norm optimal control problems for infinite dimensional linear systems, in the case in which the target is a closed ball in the state space. The fact that this proof is proof can be done in few lines is due to the fact that we take advantage from the linearity of the system and from the fact that the target set has non empty interior.

\

{\em Proof of Theorem~\ref{bb}.}
For the sake of simplicity, we use the notation $\tau^*=\tau^*(M)$. Remark that,
since $\psi\not\in B(0,r)$, we necessarily have $\tau^*>0$.
For each $\tau>0$ we set ${\cal R}_\tau =\mathbb{T}_\tau \psi+\Phi_\tau \mathcal{U}_{M}$.

As the first step of our proof, we show that
\begin{equation}\label{disjoint}
{\cal R}_{\tau^*}\cap \mathring B(0,r)=\emptyset,\;\;\mbox{\rm with}\;\; \mathring B(0,r) \;\;\mbox{\rm the interior of}\;\; B(0,r).
\end{equation}
Indeed, if (\ref{disjoint}) did not hold, then we would have  $\tilde u\in \mathcal{U}_{M}$ such that
$\|z(\tau^*;\psi,\tilde u)\|<r$. Since $\|z(0;\psi,\tilde u)\|=\|\psi\|>r$, the above, as well as
 the continuity of the map $t\mapsto z(t;\psi,\tilde u)$, implies that there exists $\tau_0\in (0,\tau^*)$ such that
$\|z(\tau_0;\psi,\tilde u)\|=r$, which contradicts the optimality of $\tau^*$.
Consequently, we have \eqref{disjoint}.

As the second step of our proof, we remark that the sets ${\cal R}_{\tau^*}$ and $\mathring B(0,r)$ are non empty,
convex and they have no common point.
Moreover, since $\mathring B(0,r)$ is open, we can apply the geometric version of
the Hahn-Banach theorem (see, for instance, in Brezis \cite[Theorem 1.6]{brezis}) to obtain that there exists
a hyperplane separating ${\cal R}_{\tau^*}$ and $B(0,r)$.
This means that there exists $\rho\in \mathbb{R}$ and $ \eta\in X$, $ \eta \not= 0$ such that
\begin{equation}\label{very_first}
\langle  \eta,w\rangle \leqslant \rho \quad(w\in {\cal R}_{\tau^*}) \qquad \mbox{and} \qquad \langle  \eta,w\rangle \geqslant \rho \quad(w\in  B(0,r)).
\end{equation}
Finally, it is obvious that $ \|\mathbb{T}_{\tau^*}\psi+\Phi_{\tau^*}u^*\|=r$. This combines with \eqref{very_first}, leads to
$$
\langle  \eta,\mathbb{T}_{\tau^*}\psi+\Phi_{\tau^*}u^*\rangle\geqslant \ \langle  \eta,w\rangle \qquad(w\in {\cal R}_{\tau^*}).
$$
The above formula clearly implies the maximum principle \eqref{PM}. The transversality condition follows directly from  \eqref{very_first} which ends the proof. \qquad $\square$

\

{\em Proof of Theorem~\ref{bb_n}.} Since $\|\mathbb{T}_\tau\psi\|>r$, we necessarily have $N^*(\tau)>0$.
Set
$$ {\cal R}_\tau := \left\lbrace z(\tau;\psi,f)\;\; |\; \; \|f\|_{L^\infty((0,\tau);U)}\leqslant N^*(\tau) \right\rbrace.
$$
The key is to show that
\begin{equation}\label{disjoint_n}
{ \cal R}_{\tau}\cap \mathring B(0,r)=\emptyset,\;\;\mbox{\rm with}\;\; \mathring B(0,r) \;\;\mbox{\rm the interior of}\;\; B(0,r).
\end{equation}
We suppose by contradiction that thers exist $\delta>0$ and $\tilde f$ in $L^\infty((0,\tau); U)$ such that
\begin{equation}\label{contradiction3.9}
\|z(\tau;\psi,\tilde f)\|\leqslant r-\delta \quad \mbox{ and} \quad \| \tilde f\|_{L^\infty((0,\tau);U)} \leqslant N^*(\tau).
\end{equation}
Since the map $\Phi_\tau: L^\infty((0,\tau);U)\rightarrow X$ is continuous,
 there exists $\alpha\in (0, 1)$ such that
$$
\|z(\tau;\psi,\alpha \tilde f)-z(\tau;\psi, \tilde f) \| \leqslant \delta.
$$
This, along with (\ref{contradiction3.9}), indicates that $ \|z(\tau;\psi, \alpha \tilde f)\| \leqslant r$.

Moreover, it follows from the first inequality in (\ref{contradiction3.9}) that $ \| \alpha \tilde f\| =\alpha N^*(\tau)< N^*(\tau)$, which combines the above inequality, contradicts to the optimality of $N^*(\tau)$ to the problem $(NP)^\tau$. Consequently, \eqref{disjoint_n} holds.

From \eqref{disjoint_n}, we can follow step by step the proof of  Theorem \ref{bb} to obtain \eqref{PM_n} and the transversality condition. This ends the proof.\qquad $\square$

\bibliographystyle{siam}
\bibliography{reference_wang}

\end{document}